\documentclass{ws-ijwmip}
\usepackage{graphicx}
\usepackage[super]{cite}
\usepackage{amsmath,amssymb,amsfonts,mathrsfs,mathtools, braket}

\DeclarePairedDelimiter\abs{\lvert}{\rvert}
\DeclarePairedDelimiter\norm{\lVert}{\rVert}

\newcommand{\R}{\mathbb{R}}
\newcommand{\C}{\mathbb{C}}
\newcommand{\N}{\mathbb{N}}
\newcommand{\Z}{\mathbb{Z}}

\newcommand{\calD}{\mathcal{D}}
\newcommand{\calO}{\mathcal{O}}

\newcommand{\calQ}{\mathcal{Q}}

\newcommand{\F}{\mathcal{F}}

\newcommand{\Q}{\mathcal{Q}}

\newcommand{\Co}{\mathrm{Co}}
\newcommand{\GL}{\mathrm{GL}}
\newcommand{\dd}{\mathrm{d}}

\newcommand{\inv}{{-1}}
\newcommand{\invT}{{-T}}

\newcommand{\what}{\widehat}
\allowdisplaybreaks

\begin{document}

\markboth{F\"uhr, Koch}{EMBEDDINGS OF SHEARLET COORBIT SPACES INTO SOBOLEV SPACES}

\catchline{}{}{}{}{}

\title{EMBEDDINGS OF SHEARLET COORBIT SPACES INTO SOBOLEV SPACES}

\author{Hartmut F\"uhr}

\address{Lehrstuhl A f\"ur Mathematik,\\
RWTH Aachen, 52056 Aachen, Germany\\
fuehr@matha.rwth-aachen.de}

\author{Ren\'e Koch}

\address{Lehrstuhl A f\"ur Mathematik,\\
RWTH Aachen, 52056 Aachen, Germany\\
rene.koch@matha.rwth-aachen.de}

\maketitle

\begin{history}
\received{(Day Month Year)}
\revised{(Day Month Year)}
\accepted{(Day Month Year)}
\published{(Day Month Year)}
\end{history}

\begin{abstract}
We investigate the existence of embeddings of shearlet coorbit spaces associated to weighted mixed $L^p$-spaces into classical Sobolev spaces in dimension three by using the description of coorbit spaces as decomposition spaces. This different perspective on these spaces enables the application of embedding results that allow the complete characterization of embeddings for certain integrability exponents, and thus provides access to a deeper understanding of the smoothness properties of coorbit spaces, and of the influence of the choice of shearlet groups on these properties. We give a detailed analysis, identifying which features of the dilation groups have an influence on the embedding behavior, and which do not. Our results also allow to comment on the validity of the interpretation of shearlet coorbit spaces as smoothness spaces.
\end{abstract}

\keywords{Shearlet coorbit spaces, embeddings, Sobolev spaces, decomposition spaces}

\ccode{AMS Subject Classification: }

\section{Introduction}



This paper is a study of approximation-theoretic properties of generalized wavelet systems arising from the action of certain matrix groups by dilation, combined with arbitrary translations. Starting with the paper by Murenzi \cite{Mu}, soon after generalized by Bernier and Taylor \cite{BeTa}, it was realized that the theory of square-integrable group representations provides access to a large variety of possible wavelet constructions, see e.g. Refs. \refcite{MR1419174,MR1633179}. An important addition to this class were the shearlets introduced for dimension two in Ref. \refcite{MR2543193}, and for higher dimensions in Ref. \refcite{MR2643586}. 
It was later realized that in dimensions $\ge 3$, several distinct choices of shearing operations can be employed, by introducing the Toeplitz shearlet construction \cite{DahHaeuTesCooSpaTheForTheToeSheTra}. A general scheme for the construction of shearlet dilation groups, which leads to a vast choice of different groups in higher dimensions, was then developed in Refs. \refcite{AlbertiEtAl2017,FuRe}. 

The interest in shearlets comes from the fact that the combination of anisotropic scaling and shearing results in a system of functions that is better equipped for the resolution of oriented singularities such as edges in images. This statement can be formalized by showing that the wavefront set of a signal can be characterized in terms of shearlet coefficients, which was first shown for two-dimensional wavelets in Ref. \refcite{KuLa}, and later extended to more general shearlet groups in arbitrary dimensions \cite{AlbertiEtAl2017}.

An alternative way of understanding how coefficient decay and smoothness of the analyzed signal are connected uses the theory of coorbit spaces. These spaces, introduced by Feichtinger and Gr\"ochenig \cite{FeiGroeI,FeiGroeII}, are based on the idea of introducing norms that quantify coefficient decay of a signal $f$ with respect to a given generalized wavelet system, and it is known that this theory applies to shearlet dilation groups in arbitrary dimensions \cite{FuRe}. Hence, each of these groups induces its own scale of coorbit spaces, defined in terms of the speed of coefficient decay. In view of the large pool of possible choices of such groups, this raises the question of analyzing and understanding coorbit spaces associated to a given shearlet dilation group, or more pointedly, understanding the influence that the choice of shearlet dilation group has on its scale of coorbit spaces. This paper can be seen as a case study for such an endeavor: We consider two families of shearlet dilation groups in 
dimension three, and analyze in a systematical fashion how coorbit spaces associated to weighted mixed $L^p$-spaces over these groups embed into Sobolev spaces. This question is interesting for several reasons. The first reason originates from the interpretation of coorbit spaces as smoothness spaces, as done, e.g., for shearlets in the introduction of Ref. \refcite{MR2643586}. This point of view seems natural given the fact that, for all shearlet dilation groups, there exist shearlet systems consisting of compactly supported, smooth functions \cite{MR2643586,FuRe}. Hence, one might expect that the elements of a coorbit space requiring fast decay of the coefficients inherit fast decay and smoothness properties from the elements of the shearlet system that efficiently approximate them. Clearly, studying the embedding behaviour of coorbit spaces into Sobolev spaces is one way of putting this general intuition to the test, and our analysis will reveal the extent to which it is justified, and how different 
features of the dilation groups influence its validity. 

Furthermore, the project of understanding the relationship between shearlet coorbit spaces and classical smoothness spaces is also motivated by work analyzing Fourier integral or pseudo-differential operators using shearlets \cite{MR2314296,MR3633063}, with a view to characterizing the mapping properties of these operators on the various function spaces. The embeddings of the type studied here fit well into this general endeavor. 

The last source of motivation that we want to mention comes from the method of proof, which largely relies on the machinery of decomposition spaces. These spaces were first introduced by Feichtinger and Gr\"obner \cite{FeichtingerGroebnerBanachSpacesOfDistributionsI} as a unified approach to both Besov and modulation spaces, with the scale of $\alpha$-modulation spaces as intermediate construction. Decomposition space applications and techniques were later extended by Borup and Nielsen \cite{MR2296727}, who pointed out (among other things) that curvelets could also be included in this setting. These ideas were further developed by the work of Voigtlaender, who introduced a powerful embedding theory between decomposition spaces of different kinds \cite{Voigtlaender2015PHD}, and of decomposition spaces into Sobolev spaces \cite{VoigtlaenderEmbeddingsOfDecSpInSobolebAndBVSPaces}. The scope of these results is truly remarkable: Among the function spaces that have a decomposition space description 
are (homogeneous) isotropic Besov spaces, or more generally, $\alpha$-modulation spaces  \cite{FeichtingerGroebnerBanachSpacesOfDistributionsI}, inhomogeneous Besov spaces \cite{Voigtlaender2015PHD} and anisotropic Besov spaces (both homogeneous and inhomogeneous) \cite{CheshmavarFuehr}. Another class of examples, which is of particular relevance for this paper, are the coorbit spaces associated to general dilation groups, introduced in full generality in Ref. \refcite{FuehCooSpaAndWavCoeDecOveGenDilGro}, and identified as decomposition spaces in Refs. \refcite{MR3345605,Voigtlaender2015PHD}. In particular, all of the previously mentioned shearlet coorbit spaces fall in this category. 

Hence the embedding theory developed in  Ref. \refcite{VoigtlaenderEmbeddingsOfDecSpInSobolebAndBVSPaces} is applicable to our problem, and our paper is both a sample application of the methods developed in the cited paper, and an illustration of the remarkable power of these methods. Prior to the work of Voigtlaender, an analysis of comparable depth and scope had simply been out of reach. 

\subsection*{Overview and summary of the paper}

The paper is structured as follows: Sections \ref{sect:WT} through \ref{sect:Embeddings} introduce the objects and results necessary to formulate and prove our main result. The class of shearlet transforms that we study is introduced in Section \ref{sect:WT}. Coorbit spaces, and their decomposition space description, are introduced and explained in Section \ref{sect:CoorbitAndDecompositionSpaces}. Voigtlaender's embedding result is then formulated in Section \ref{sect:Embeddings}. Theorem \ref{thm:SobolevEmbedding} shows that a decomposition space of the type $\mathcal{D}(\mathcal{P},L^p,\ell^q_v)$ embeds into a Sobolev space if a certain sequence, that is explicitly derived from the data $\mathcal{P},p,q,v$ entering the definition of the decomposition space, is summable in a suitable sense. For summability indices $p,q \le 2$, this statement in fact becomes an equivalence. The systematic application of this result to the setting of shearlet coorbit spaces is then the subject of Section 
\ref{section:Embeddings_shearlets}. 
The main technical results of our paper are the Theorems \ref{Cor:EmbeddingStandardLong}, containing a precise and exhaustive 
characterization of embeddings into Sobolev spaces for coorbit spaces associated to the standard shearlet groups in dimension three, and Theorem \ref{cor:ResultatToeplitz1}, which formulates an analog for Toeplitz shearlet groups. For two-dimensional shearlet groups, this analysis had been performed in Ref. \refcite{VoigtlaenderEmbeddingsOfDecSpInSobolebAndBVSPaces}, but for the three-dimensional cases, the results are completely new. They are also substantially more complicated than those for the two-dimensional case, due to the additional parameter describing the anisotropic scaling in the third dimension, and the additional variation in the shearing subgroup of the Toeplitz shearlet group. 

In particular, while the Theorems \ref{Cor:EmbeddingStandardLong} and \ref{cor:ResultatToeplitz1} contain essentially complete information, the interpretation of these results, say with a view to investigating the smoothness space interpretation of coorbit spaces, and the influence of the dilation group on this issue, becomes a separate nontrivial problem, which we address in Section \ref{Consequences}. We restrict the discussion to certain pertinent subcases, and analyze more closely the influence of the different components of the shearlet dilation group on the embedding behavior. Specifically, we investigate the role of the exponents describing the scaling subgroup on the one hand, and the choice of shearing subgroup on the other.  

This allows to draw the following conclusions: For the important subcase of coorbit spaces of the kind $Co(L^p)$, where $0 < p < 2$, {\em there exists no embedding into Sobolev spaces with nontrivial smoothness exponent $k>0$}, regardless of the choice of shearlet dilation group; see Corollary \ref{cor:no_embedding}. Elements of these spaces can be understood as functions in $L^2$ with a non-trivial approximation rate -- in the $L^2$-norm -- with respect to any discrete shearlet system obtained by discretizing the continuous shearlet transform; see e.g. the discussion in Ref. \refcite[Section 1.1]{MR3452925}. Our observation makes clear that this type of decay alone does not guarantee smoothness. We then resort to (mostly) analyzing embeddings of $Co(L^p_v)$ for suitable weights, and $0<p<2$. Here, it turns out that the embedding behavior depends on certain features of the shearlet dilation groups, and is independent of others. More precisely, the {\em shearing subgroup} has no influence 
(Corollary \ref{Cor:EmbeddingBehavior1}), whereas the {\em scaling subgroup} is critically influential (Theorem \ref{Cor:SameEmbeddingStandard}). The fact that the shearing subgroup has no influence is interesting because the coorbit spaces associated to different shearing subgroups do {\em not} coincide, by the results in Ref. \refcite{RKDoktorarbeit}. Thus the embedding behaviour into Sobolev spaces does not allow to distinguish different scales of shearlet coorbit spaces. This fact could probably have been expected, but it has been brought to light and rigourously established by our analysis. 

As a further byproduct of our analysis, we obtain that {\em anisotropic} scaling is required to guarantee the existence of embeddings into Sobolev spaces with nontrivial smoothness parameter. This means that using multiples of the identity operator as the scaling subgroup never works, see Corollary \ref{cor:embedding_2} and the following Remark \ref{rem:embedding_2} (2). On the one hand, this observation is slightly surprising, as the target spaces $W^{k,p}$ of the embedding results exhibit no anisotropies. On the other hand, it is well in line with the fact that anisotropy is needed for the study of singularities, such as the wavefront set, via the decay behaviour of shearlet transforms. Here, anisotropic scaling is generally necessary \cite{MR3547710}, and -- with additional restrictions -- also sufficient \cite{AlbertiEtAl2017}. 

Finally, we study for which groups the smoothness of elements of coorbit spaces improve as the decay requirements imposed by the coorbit spaces become more restrictive. This can be done by asking how the best possible parameter $k$ in the embedding $Co(L_v^p) \hookrightarrow W^{k,q}$ scales as $p$ goes to zero. Again, our results show that this can be attributed to properties of the scaling subgroup alone; see Corollary \ref{cor:embedding_3} and the subsequent Remark. 

\section{Generalized Wavelet Transform and Shearlet Groups}\label{section:ShearletGroups}

\label{sect:WT}
In this section, we recall basic definitions underlying the continuous wavelet transform and generalized shearlet dilation groups.

\subsection{Generalized Wavelet Transform}
For a closed matrix group $H\leq \GL(\R^d)$, which we also call \textit{dilation group} in the following, we define the group $G:=\R^d\rtimes H$ generated by dilations with elements of $H$ and translations with the group law $(x,h)\circ (y,g) := (x+hy, hg)$. We denote integration with respect to a left Haar measure on $H$ with $\dd h$, the associated left Haar measure on $G$ is then given by $d(x,h)=\abs{\det h}^{-1}\dd x \dd h$. The Lebesgue spaces on $G$ are always defined through integration with respect to a Haar measure. The group $G$ acts on the space $\mathrm{L}^2(\R^d)$ through the \textit{quasi-regular representation} $\pi$ defined by $[\pi(x,h)f](y):=\abs{\det h}^{-1/2} f(h^\inv (y-x))$ for $f\in \mathrm{L}^2(\R^d)$. The \textit{generalized continuous wavelet transform (with respect to $\psi\in \mathrm{L^2(\R^d)}$)} of $f$ is then given as the function
$W_\psi f:G \to \C: (x,h)\mapsto \braket{\psi, \pi(x,h)f}.$ Important properties of the map $W_\psi:f\mapsto W_\psi f$ depend on $H$ and the chosen $\psi$. If the quasi-regular representation is \textit{square-integrable}, which means that there exists a $0\neq \psi$ with $W_\psi \psi \in \mathrm{L}^2(G)$, and irreducible, then we call $H$ \textit{admissible} and the map $W_\psi:\mathrm{L}^2(\R^d) \to \mathrm{L}^2(G)$ is a multiple of an isometry, which gives rise to the (weak-sense)
    inversion formula
    \begin{equation} \label{eqn:waverec}
     f = \frac{1}{C_\psi}\int_G W_\psi f(x,h) \pi(x,h) \psi \dd(x,h) ~,
    \end{equation}
i.e., each $f \in \mathrm{L}^2(\mathbb{R}^d)$  is a continuous superposition of the wavelet system. According to results in Refs. \refcite{MR1419174}, \refcite{MR2652610}, the admissibility of $H$ can be characterized by the \textit{dual action} defined by $G\mapsto \R^d, (\xi,h)\mapsto p_\xi(h):=h^\invT \xi$, where $p_\xi$ denotes the associated orbit map. In fact, $H$ is admissible iff the dual action has a single open orbit $\mathcal{O}:=H^\invT \xi_0\subset \R^d$ of full measure for some $\xi_0\in \R^d$ and additionally the isotropy group $H_{\xi_0}:=\Set{h :p_{\xi_0}(h)=\xi_0}$ is compact.

\subsection{Shearlet Groups}

In order to state the definition given in Ref. \refcite{AlbertiEtAl2017}, we use the notation of $\mathfrak{gl}(\R^d)$ for the set of all $d \times d$-matrices and  let 
$\exp: \mathfrak{gl}(\R^d) \to \mathrm{GL}(\R^d)$
be the exponential map defined by the  series
$\exp(A) := \sum_{k=0}^{\infty}\frac{A^k}{k!}$
for every $A\in\mathfrak{gl}(\R^d)$. We consider convergence of this series with respect to the norm $\norm{A}_{\mathrm{op}}:=\sup_{\abs{x}\leq 1}\abs{Ax}$. Furthermore, we denote with $T(\R^d)$ the set of upper triangular $d\times d$-matrices with one on their diagonals.

\begin{definition}[Ref. \refcite{AlbertiEtAl2017} Definition 1.]
Let $H\subset \mathrm{GL}(\R^d)$ be a closed, admissible dilation group. The group $H$ is called \textit{generalized shearlet dilation group}\index{generalized shearlet dilation group} if there exist two closed subgroups 
$S, D \subset \mathrm{GL}(\R^d)$
such that
\begin{enumerate}[i)]
\item $S$ is a connected abelian subgroup of $T(\R^d)$,
\item $D=\Set{\exp(rY) | r\in \R}$ is a one-parameter group, where $Y\in \mathfrak{gl}(\R^d)$ is a diagonal matrix and
\item every $h\in H$ has a unique representation as $h=\pm ds$ for some $d\in D$ and $s\in S$.
\end{enumerate}
$S$ is called the \textit{shearing subgroup of $H$}\index{shearing subgroup}, $D$ is called the \textit{scaling subgroup of $H$}\index{scaling subgroup}, and $Y$ is called the \textit{infinitesimal generator of $D$}.
\end{definition}

We denote the canonical basis of $\R^d$ with $e_1,\ldots, e_d$ and the identity matrix in $\mathrm{GL}(\R^d)$ with $E_d$ or just $E$ if the dimension is clear from the context. The next result contains information about the structure of shearing subgroups. All generalized shearlet dilation groups in dimension $d$ share the same open dual orbit and isotropy group.

\begin{lemma}[Ref. \refcite{AlbertiEtAl2017} Proposition 11.]\label{lem:OrbitShearletGroups}
For a generalized shearlet dilation group $H$, the unique open dual orbit of $H$ is given by $\mathcal{O}=\R^* \times \R^{d-1}$ and the isotropy group of $\xi\in \mathcal{O} $ with respect to the dual action is given by $H_\xi=\Set{E_d}$.
\end{lemma}

Now, we introduce some concrete classes of shearlet groups, which we will further investigate in the next sections. The class of standard shearlet groups
\begin{align*}
H^{\lambda_1, \lambda_2} 
&:=
\Set{\epsilon
\begin{pmatrix}
a &	 ab 			& ac 			\\
0 &	 a^{\lambda_1} 	& 0		\\
0 &	  0	& a^{\lambda_2}
\end{pmatrix}
|
\begin{array}{l}
a>0, \\
b, c\in \R,\\
\epsilon\in \Set{\pm 1}
\end{array}
}
\shortintertext{for $\lambda_1, \lambda_2 \in \R$ and the class of Toeplitz shearlet groups}
H^\delta
&:=
\Set{\epsilon
\begin{pmatrix}
a &	 ab 			& ac 			\\
0 &	 a^{1-\delta} 	& a^{1-\delta}b		\\
0 &	  0	& a^{1-2\delta}
\end{pmatrix}
|
\begin{array}{l}
a>0,\\
b, c\in \R,\\
\epsilon\in \Set{\pm 1}
\end{array}
}
\end{align*}
for $\delta\in \R$. 

\begin{remark}
In dimension three, these are the only possible generalized shearlet dilation groups (see Ref. \refcite{AlbertiEtAl2017} remark 19).
\end{remark}

\section{Coorbit Spaces and Decompositions Spaces}\label{sect:CoorbitAndDecompositionSpaces}

Coorbit spaces are defined in terms of decay behavior of generalized wavelet transforms. 
    To give a precise definition, we introduce weighted mixed $\mathrm{L}^p$-spaces on $G$, denoted by $\mathrm{L}^{p,q}_v(G)$ . By definition, this space is the set of functions
    \begin{align*}
        \left\{ f:G\to \mathbb{C} : \int_H\left( \int_{\mathbb{R}^d} \left| f(x,h) \right|^p v(x,h)^p \dd x \right)^{q/p}\frac{\dd h}{|\det(h)|} <\infty \right\},
    \end{align*}
    with natural (quasi-)norm $\Vert \cdot\Vert_{\mathrm{L}^{p,q}_v}$. This definition is valid for $0< p,q <\infty$, for $p=\infty$ or $q=\infty$ the essential supremum has to be taken at the appropriate place instead. The function $v:G\to \mathbb{R}^{>0}$ is a weight function that fulfills the condition $v(ghk)\leq v_0(g)v(h)v_0(k)$ for some submultiplicative weight $v_0$. If the last condition is satisfied, we call $v$ left- and right moderate with respect to $v_0$. Thus, the expression $\Vert W_\psi f\Vert_{\mathrm{L}^{p,q}_v}$ can be read as a measure of wavelet coefficient decay of $f$. We will exclusively consider weights which only depend on $H$.  The coorbit space $\mathrm{Co}\left(\mathrm{L}^{p,q}_v(\mathbb{R}^d\rtimes H)\right)$ is then defined as the space
    \begin{align}\label{def:Coorbit}
        \left\{ f\in \mathcal{(H}^1_w)^\neg : W_\psi f \in W(\mathrm{L}^{p,q}_v(G))\right\}
    \end{align}
    for some suitable wavelet $\psi$ and some control weight $w$ associated to $v$. The space $(\mathcal{H}^1_w)^\neg$ denotes the space of antilinear functionals on 
$
        \mathcal{H}^1_w :=\left\{ f\in \mathrm{L}^2(\mathbb{R}^d): W_\psi f \in \mathrm{L}^1_w(G)\right\}
$
    and $W(Y)$ for a function space $Y$ on $G$ denotes the Wiener amalgam space defined by
$
    W_Q(Y):=\{f\in \mathrm{L}^\infty_{\text{loc}}(G) | M_Qf\in Y\}
 $
    with quasi-norm $\|f\|_{W_Q(Y)}:=\|M_Qf\|_Y$ for $f\in  W_Q(Y)$, where the \textit{maximal function} $M_Qf$ for some suitable unit neighborhood $Q\subset G$ is $
    M_Qf:G\to [0,\infty],\ x\mapsto \operatorname*{ess\ sup}_{y\in xQ}|f(y)|.$
   
   The appearance of the Wiener amalgam space in (\ref{def:Coorbit}) is necessary to guarantee consistently defined quasi-Banach spaces in the case $\{p,q\}\cap (0,1)\neq \emptyset$, see Ref. \refcite{RauCooSpaTheForQuaBanSpa} and Ref. \refcite{Voigtlaender2015PHD}. In the classical coorbit theory for Banach spaces, which was developed in Refs. \refcite{FeiGroeI}, \refcite{FeiGroeII}, the Wiener amalgam space can be replaced by the simpler space $\mathrm{L}^{p,q}_v(G)$, see  Ref. \refcite{RauCooSpaTheForQuaBanSpa}.
   
  Many useful properties of these spaces are known and hold in the quasi-Banach space case as well as in the Banach space case. The most prominent examples of coorbit spaces associated to generalized wavelet transforms are the homogeneous Besov spaces and the modulation spaces. However, each shearlet group gives rise to its scale of coorbit spaces, as well; see Refs.  \refcite{DahHaeuTesCooSpaTheForTheToeSheTra}, \refcite{FuehCooSpaAndWavCoeDecOveGenDilGro}, \refcite{KLShearlets}.

    The starting point for the definition of decomposition spaces is the notion of an {\em admissible covering} $\mathcal{Q}=(Q_i)_{i\in I}$ of some open set $\mathcal{O} \subset \mathbb{R}^d$ (Ref. \refcite{FeichtingerGroebnerBanachSpacesOfDistributionsI}) which is a family of nonempty sets such that
    \begin{enumerate}[i)]
        \item $\bigcup_{i\in I} Q_i = \mathcal{O}$ and
        \item $\sup_{i\in I} |\{j\in I: Q_i \cap Q_j \neq \emptyset\}|<\infty$.
    \end{enumerate} 
    The main tool for the localization is a special partition of unity $\Phi=(\varphi_i)_{i\in I}$ subordinate to $\mathcal{Q}$, also called $\mathrm{L}^p$-BAPU (bounded admissible partition of unity), with the following properties
    \begin{enumerate}[i)]
        \item $\varphi_i \in C_c^\infty(\mathcal{O})\quad \forall i\in I$,
        \item $\sum_{i\in I} \varphi_i(x)=1 \quad \forall x\in \mathcal{O}$,
        \item $\varphi_i(x)=0$ for $x\in \mathbb{R}^d\setminus Q_i$ and $i\in I$,
        \item if $1\leq p \leq \infty$: $\sup_{i\in I}\Vert \mathcal{F}^{-1} \varphi_i\Vert_{\mathrm{L}^1}<\infty$\\
                if $0<p<1$:\quad $\sup_{i\in I}|\det(T_i)|^{\frac{1}{p}-1}\Vert \mathcal{F}^{-1} \varphi_i\Vert_{\mathrm{L}^p}<\infty$,
    \end{enumerate}
    where we have to further assume in the case $0<p<1$ that the covering $\mathcal{Q}$ has the structure
    $Q_i = T_i Q + b_i$ with $T_i\in \mathrm{GL}(\mathbb{R}^d)$, $b_i\in \mathbb{R}^d$ and an open, precompact set $Q$ ($\mathcal{Q}$ is then called a {\em structured admissible covering}). The definition of decomposition spaces requires one last ingredient, namely a weight  $(u_i)_{i\in I}$ such that there exists $C>0$ with $u_i \leq C u_j$ for all $i,j \in I: Q_i \cap Q_j \neq \emptyset$, a weight with this property is also called {\em $\mathcal{Q}$-moderate}. The interpretation of this property is that the value of $(u_i)_{i\in I}$ is comparable for indices corresponding to sets which are "close" to each other. Finally, we define the decomposition space with respect to the covering $\mathcal{Q}$ and the weight $(u_i)_{i \in I}$ with integrability exponents $0<p,q\leq \infty$ as
    \begin{align}
        \mathcal{D}(\mathcal{Q}, \mathrm{L}^p, \ell^q_u):=\{f\in \mathcal{D}'(\mathcal{O}): \Vert f\Vert_{\mathcal{D}(\mathcal{Q}, \mathrm{L}^p, \ell^q_u)}< \infty\}
    \end{align}
    for 
    \begin{align}
        \Vert f\Vert_{\mathcal{D}(\mathcal{Q}, \mathrm{L}^p, \ell^q_u)}:=
        \left\Vert\left(u_i \cdot \Vert \mathcal{F}^{-1}(\varphi_i f) \Vert_{\mathrm{L}^p(\mathbb{R}^d)} \right)_{i\in I}\right\Vert_{\ell^q(I)}.
    \end{align}
    As the notation suggests, the decomposition spaces are independent of the precise choice of $\Phi$  (Ref. \refcite{Voigtlaender2015PHD} Corollary 3.4.11). 

In order to describe coorbit spaces as decomposition spaces, we need to associate a covering of the frequencies to a given dilation group. This is done using the {\em dual action} $H \times \mathbb{R}^d \ni (h, \xi) \mapsto h^{-T} \xi$.

We then pick a {\em well-spread} family in $H$, i.e. a family of elements $(h_i)_{i\in I}$ with the properties
    \begin{enumerate}[i)]
        \item there exists a relatively compact neighborhood $U\subset H$ of the identity such that $\bigcup_{i\in I}h_i U= H$ -- we say $(h_i)_{i\in I}$ is \textit{$U$-dense} in this case -- and
        \item there exists a neighborhood $V\subset H$ of the identity such that $h_iV \cap h_jV = \emptyset$ for $i\neq j$ -- we say $(h_i)_{i\in I}$ is \textit{$V$-separated} in this case.
    \end{enumerate}
    The  \textit{dual covering induced by $H$} is then given by the family  $\mathcal{Q}=(Q_i)_{i\in I}$, where $Q_i = p_{\xi_0}(h_i U)$ for some $\xi_0$ with $H^{-T}\xi_0=\mathcal{O}$. It can be shown that well-spread families always exist, and that the induced covering is indeed an admissible covering in the sense of decomposition space theory, for which $\mathrm{L}^p$-BAPUs exist according to Ref. \refcite{Voigtlaender2015PHD}. Furthermore, there always exist induced coverings consisting of open and connected sets, an additional feature which can facilitate the investigations in some cases, see Ref. \refcite{RKDoktorarbeit} Corollary 2.5.9.
    
     There always exists a discretization of the weight $v$, which enables a decomposition space description of the coorbit space. 
     
\begin{definition}[Ref. \refcite{Voigtlaender2015PHD} Definition 4.5.3.]\label{def:DecompositionWeight}
For $q\in (0,\infty]$ and a weight $v:H\to (0,\infty)$, we define the weight
$
v^{(q)}:H\to (0,\infty),\ h\mapsto \abs{\det(h)}^{\frac{1}{2} - \frac{1}{q}} v(h).
$ Here, we set $\frac{1}{\infty}:=0$. 
\end{definition}

\begin{theorem}[Ref. \refcite{Voigtlaender2015PHD} Theorem 4.6.3]\label{thm: FourierIsoCoorbitDecSpaces}
Let $\mathcal{Q}$ be a covering of the dual orbit $\mathcal{O}$ induced by $H$, $0<p,q\leq\infty$ and $u=(u_i)_{i\in I}$ a suitable weight, then the Fourier transform
$
       \mathcal{F}: \mathrm{Co}\left(\mathrm{L}^{p,q}_v(\mathbb{R}^d\rtimes H)\right) \to \mathcal{D}(\mathcal{Q}, \mathrm{L}^p, \ell^q_u)
$
    is an isomorphism of (quasi-) Banach spaces. The weight $(u_i)_{i\in I}$ can be chosen as $u_i:=v^{(q)}(h_i)$, where $(h_i)_{i\in I}$ is the well-spread family used in the construction of $\mathcal{Q}$ and we call such a weight a $\mathcal{Q}-$discretization of $v$.
\end{theorem}

In order to apply the embedding results in Ref. \refcite{VoigtlaenderEmbeddingsOfDecSpInSobolebAndBVSPaces}, we need explicitly given well-spread families in standard and Toeplitz shearlet groups. 

    \begin{lemma}\label{lem:StandardWellSpread}
Define
\begin{align*}
B^{\lambda_1, \lambda_2}_{n, m_1, m_2, \epsilon} 
:=
\epsilon
\begin{pmatrix}
2^n &	 m_1 2^n 			& m_2 2^n	\\
0 &	 2^{n\lambda_1} 	& 0		\\
0 &	  0	& 2^{n\lambda_2}
\end{pmatrix}
\in H^{\lambda_1, \lambda_2}
\end{align*}
for $n, m_1, m_2 \in \Z$ and $\epsilon\in \{\pm 1\}$. Then the family
$\Gamma^{\lambda_1, \lambda_2} := \left(B^{\lambda_1, \lambda_2}_{n, m_1, m_2, \epsilon} \right)_{(n,m_1,m_2,\epsilon)\in I}$ for $I:=\Z^3\times \{\pm 1\}$
is $U^{\lambda_1, \lambda_2}$ -dense and -separated, with
\begin{align*}
U^{\lambda_1, \lambda_2} &:= 
\Set{
\begin{pmatrix}
\alpha 	&	 \alpha\beta 			& \alpha\gamma 			\\
0 		&	 \alpha^{\lambda_1} 	& 0						\\
0 		&	  0						& \alpha^{\lambda_2}
\end{pmatrix}
|
\begin{array}{cll}
\frac{2}{3} &< \alpha &\leq\frac{4}{3}, \\
-\frac{\alpha^{\lambda_1}}{2} &< \alpha\beta &\leq \frac{\alpha^{\lambda_1}}{2}, \\
-\frac{\alpha^{\lambda_2}}{2} &< \alpha\gamma &\leq \frac{\alpha^{\lambda_2}}{2}
\end{array}
}
\subset H^{\lambda_1, \lambda_2}.\\
\hfill
\end{align*}
Furthermore, this well-spread family induces a covering $\mathcal{C}^{\lambda_1, \lambda_2}$ of the associated dual orbit. For later reference, we define $A^{\lambda_1, \lambda_2}_{n, m_1, m_2, \epsilon} 
:=\left(B^{\lambda_1, \lambda_2}_{n, m_1, m_2, \epsilon}\right)^{-T}$.
\end{lemma}

\begin{lemma}\label{lem:ToeplitzWellSpread}
Define
\begin{align*}
B^{\delta}_{n, m_1, m_2, \epsilon} 
:=
\epsilon
\begin{pmatrix}
2^n &	 m_1 2^n 			& m_2 2^n	\\
0 &	 2^{n(1-\delta)} 	& m_12^{n(1-\delta)}		\\
0 &	  0	& 2^{n(1-2\delta)}
\end{pmatrix}
\in H^{\delta}
\end{align*}
for $n, m_1, m_2 \in \Z$ and $\epsilon\in \{\pm 1\}$. Then the family $\Gamma^{\delta} := \left(B^{\delta}_{n, m_1, m_2, \epsilon} \right)_{(n,m_1,m_2,\epsilon)\in I}$
is $U^{\delta}$ -dense and -separated, with
\begin{align*}
U^{\delta} &:= 
\Set{
\begin{pmatrix}
\alpha 	&	 \alpha\beta 			& \alpha\gamma 			\\
0 		&	 \alpha^{1-\delta} 	& 	\alpha^{1-\delta}\beta				\\
0 		&	  0						& \alpha^{1-2\delta}
\end{pmatrix}
|
\begin{array}{cll}
\frac{2}{3} &< \alpha &\leq\frac{4}{3}, \\
-\frac{\alpha^{1-\delta}}{2} &< \alpha\beta &\leq \frac{\alpha^{1-\delta}}{2}, \\
-\frac{\alpha^{1-2\delta}}{2} &< \alpha\gamma &\leq \frac{\alpha^{1-2\delta}}{2}
\end{array}
}
\subset H^{\delta}.\\
\hfill
\end{align*}
Furthermore, this well-spread family induces a covering $\mathcal{C}^{\delta}$ of the associated dual orbit. For later reference, we define $A^{\delta}_{n, m_1, m_2, \epsilon} 
:=\left(B^{\delta}_{n, m_1, m_2, \epsilon}\right)^{-T}$.
\end{lemma}

For the following investigation of the existence of embeddings of the associated shearlet coorbit into Sobolev spaces, it is not necessary to have an explicit description of these induced coverings, but the reader can find one in Ref. \refcite{RKDoktorarbeit}.

\section{Embeddings into Sobolev spaces}
\label{sect:Embeddings}

Our goal in this section is to study the embedding behavior of the coorbit spaces  $\Co\left(\mathrm{L}^{p,q}_m(\R^3\rtimes H^\delta)\right)$ and $\Co\left(\mathrm{L}^{p,q}_m(\R^3\rtimes H^{\lambda_1, \lambda_2})\right)$ associated to three dimensional shearlet groups into Sobolev spaces for $p,q \in (0,\infty]$ and a specific class of weights $m$ on $\GL(\R^3)$.

This section is based on methods developed in Ref. \refcite{VoigtlaenderEmbeddingsOfDecSpInSobolebAndBVSPaces}. There, the two dimensional case was considered and characterized to a large extent. Similar questions related to the embedding of (subspaces) of certain shearlet coorbit spaces into classical smoothness spaces were also investigated in Ref. \refcite{DahlkeShearletCoorbitSpacesCompactlySupported} (also in two dimensions).

At first, we introduce the necessary tools, in particular, a definition of Sobolev spaces for integrability exponents $0<p<1$, and another type of partition of unity. Afterwards, we explain what we precisely mean by an embedding of a coorbit space into a Sobolev space, which will depend on the identification of the coorbit space with a suitable decomposition space, and present the result in Ref. \refcite{VoigtlaenderEmbeddingsOfDecSpInSobolebAndBVSPaces} which is a crucial tool in this chapter.

In the following sections, we apply the result and will compare coorbit spaces associated to different groups with regards to their embedding behavior. Surprisingly, the embedding behavior of shearlet groups in dimension three into Sobolev spaces is determined by the scaling subgroup of the group, which means the shearing part has no influence on the embedding behavior for parameters in a suitable range and the class of weights we consider.

We will also see how the group parameters influence the smoothness of the elements in the associated coorbit spaces.

We start by giving a definition of Sobolev spaces, which is completely classical in the Banach space case. The definition in the quasi-Banach space case used in Ref. \refcite{VoigtlaenderEmbeddingsOfDecSpInSobolebAndBVSPaces} is inspired by the definition by Peetre \cite{PeetreSobolev} and the resulting spaces exhibit quite unexpected properties. For example, the definition as tuple of functions is motivated by the fact that $(\partial^\alpha f)_{\alpha\in \N_0^d, \abs{\alpha}_1 \leq k}\mapsto f_0$ is not injective in general. Since we heavily rely on the results in Ref. \refcite{VoigtlaenderEmbeddingsOfDecSpInSobolebAndBVSPaces} , we adhere to the definition employed there.

\begin{definition}[Ref. \refcite{VoigtlaenderEmbeddingsOfDecSpInSobolebAndBVSPaces} Subsection 2.1]
We define
\begin{align*}
W^{k,q}(\R^d):=\Set{f\in \mathrm{L}^q(\R^d) | 
\begin{array}{l}
\partial^\alpha f \in \mathrm{L}^q(\R^d) \text{ for all }\\
\alpha\in \N_0^d,\ |\alpha|_1 \leq k
\end{array}
}
\end{align*}
for $q\in [1,\infty]$, where $\partial^\alpha f$ denotes the weak partial derivative of $f$ and for $0<q<1$, let $W^{k,q}(\R^d)$ be the closure of
\begin{align*}
W^{k,q}_*(\R^d):=\Set{ (\partial^\alpha f)_{\alpha\in \N_0^d, \abs{\alpha}_1 \leq k} | 
\begin{array}{l}
f\in C^\infty(\R^d),\ \partial^\alpha f \in \mathrm{L}^q(\R^d)\\
\text{for all }\alpha\in \N_0^d,\ |\alpha|_1 \leq k
\end{array}
}
\end{align*}
in the product $\prod_{\alpha\in \N_0^d,\ |\alpha|_1 \leq k}L^q(\R^d)$.
\end{definition}

We will define suitable differentiation operators on decomposition spaces by resorting to special partitions of unity. 

\begin{definition}[Ref. \refcite{VoigtlaenderEmbeddingsOfDecSpInSobolebAndBVSPaces} Definition 2.4.]
Let $\calQ=(T_iQ+b_i)_{i\in I}$ be a structured admissible covering of some open set $\calO \subset  \R^d$ and let $(\phi_i)_{i\in I}$ be a partition of unity subordinate to $\calQ$ with $\phi_i \in C_c^\infty (\calO)$. The \textit{normalized version} of $\phi_i$ is given by
$\phi_i^\#:\R^d \to \C,\ x\mapsto \phi_i(T_ix+b_i)$
for $i\in I$. Additionally, we say that $(\phi_i)_{i\in I}$ is a \textit{regular partition of unity}\index{regular partition of unity} if
$\sup_{i\in I}\norm*{\partial^\alpha \phi_i^\#}_\mathrm{sup} < \infty$
for all $\alpha\in \N_0^d$.
\end{definition}

The connection between regular partitions of unity and $L^p$-BAPUs as well as their existence for our preferred type of covering are given by the next lemmata.

\begin{lemma}[Ref. \refcite{VoigtlaenderEmbeddingsOfDecSpInSobolebAndBVSPaces} Corollary 2.7.]
If $(\phi_i)_{i\in I}$ is a regular partition of unity subordinate to a structured admissible covering $\calQ$, then $(\phi_i)_{i\in I}$ is an $L^p$-BAPU for every $p\in (0,\infty]$.
\end{lemma}

\begin{lemma}[Ref. \refcite{VoigtlaenderEmbeddingsOfDecSpInSobolebAndBVSPaces} Theorem 2.8.]
Every structured admissible covering admits a subordinate regular partition of unity.
\end{lemma}

These preparations allow us to define differentiation operators on decomposition spaces and to specify what we mean by an embedding of a decomposition space or coorbit space into a Sobolev space.

\begin{definition}\label{def:EmbeddingSobolev}
Let $\Q=(Q_i)_{i\in I}$ be a structured admissible covering of some open set $\mathcal{O}\subset\R^d$ and $(u_i)_{i\in I}$ a $\Q$-moderate weight on $I$. For $p,r\in (0,\infty]$, let $(\phi_i)_{i\in I}$ be a regular partition of unity for $\calQ$.
\begin{enumerate}[i)]
\item We say that $\calD(\calQ, \mathrm{L}^p, \ell^r_u)$ \textit{admits a partial differential operator} (with respect to $q$) for $\alpha\in \N_0^k$ with $k\in \N_0$ if the map
\begin{align*}
 & \partial_*^\alpha: \calD(\calQ, \mathrm{L}^p, \ell^r_u) &&\to&& \mathrm{L}^q(\R^d)\\
 & \qquad f &&\mapsto&& \sum_{i\in I}\partial^\alpha\left[\F^\inv(\phi_i f)\right]
\end{align*}
is well-defined, bounded, with unconditional convergence of the series.

\item We write  $\calD(\calQ, \mathrm{L}^p, \ell^r_u) \hookrightarrow W^{k,q}(\R^d)$ if $\calD(\calQ, \mathrm{L}^p, \ell^r_u)$ admits a partial differential operator in the sense of i) and 
\begin{align*}
\iota^{(k)}_q: &\calD(\calQ, \mathrm{L}^p, \ell^r_u) &&\to&& W^{k,q}(\R^d)\\
&f &&\mapsto&& \partial_*^0 f
\end{align*}
is well-defined, bounded and injective for $q\geq 1$.

In the case $0<q<1$, we require that the map
\begin{align*}
\iota^{(k)}_q: &\calD(\calQ, \mathrm{L}^p, \ell^r_u) &&\to&& W^{k,q}(\R^d)\\
& f &&\mapsto&& (\partial^\alpha_* f )_{\abs{\alpha}_1 \leq k}
\end{align*}
is well-defined and bounded.

\item We write $\Co\left(\mathrm{L}^{p,r}_m(\R^d\rtimes H)\right)\hookrightarrow W^{k,q}(\R^d)$ for $r, q\in (0,\infty]$, where $H$ is an admissible group with dual orbit $\calO$ and $m$ is a weight on $H$ that is right moderate with respect to a locally bounded weight on $H$ if the associated isomorphic decomposition space in the sense of Theorem \ref{thm: FourierIsoCoorbitDecSpaces} satisfies
$\calD(\calQ, \mathrm{L}^p, \ell^r_u)\hookrightarrow W^{k,q}(\R^d).$
\end{enumerate}
\end{definition}

\begin{remark}
\begin{enumerate}[i)]
\item This definition is inspired by Theorem 3.4. and Corollary 4.5. in Ref. \refcite{VoigtlaenderEmbeddingsOfDecSpInSobolebAndBVSPaces}. We had to adapt the definition of the partial differential operator slightly (by changing $\hat{f}$ to $f$) because Ref. \refcite{VoigtlaenderEmbeddingsOfDecSpInSobolebAndBVSPaces} works with space-side decomposition spaces. Despite this change, all the results and proofs in Ref. \refcite{VoigtlaenderEmbeddingsOfDecSpInSobolebAndBVSPaces} can be applied to our setting by just interchanging $\hat{f}$ to $f$ at the appropriate place.

\item A motivation for this definition is the fact that if $\calD(\calQ, \mathrm{L}^p, Y)$ admits a partial differential operator $\partial_*^\alpha$, then $\partial_*^\alpha f = \partial^\alpha\left(\F^\inv f\right)$
for all $f\in \calD(\calQ, \mathrm{L}^p, Y) \cap C_c^\infty(\calO)$, according to Ref. \refcite{VoigtlaenderEmbeddingsOfDecSpInSobolebAndBVSPaces} Theorem 3.4. Furthermore, $\calD(\calQ, \mathrm{L}^p, Y) \hookrightarrow W^{k,q}(\R^d)$ implies $\partial^\alpha (\iota_q^{(k)} f) = \partial_*^\alpha f$ for $f\in \calD(\calQ, \mathrm{L}^p, Y)$ and $|\alpha|_1\leq k$, according to Corollary 3.5 in Ref. \refcite{VoigtlaenderEmbeddingsOfDecSpInSobolebAndBVSPaces}. Hence, $\Co(\mathrm{L}^{p,q}_m(\R^d\rtimes H)\hookrightarrow W^{k,q}(\R^d)$ implies 
\begin{align*}
&\Co\left(\mathrm{L}^{p,q}_m(\R^d\rtimes H)\right)&\xrightarrow{\F} &\calD(\calQ, \mathrm{L}^p, \ell^q_u) &&\xrightarrow{\partial_*^\alpha} \mathrm{L}^q(\R^d)\\
&f &\xmapsto{\F} &\hat{f} &&\xmapsto{\partial_*^\alpha}\partial_*^\alpha \hat{f} = \partial^\alpha f
\end{align*}

for $f\in  \Co(\mathrm{L}^{p,q}_m(\R^d\rtimes H) \cap \F^\inv(C_c^\infty(\calO))$ and $|\alpha| \leq k$, where we used that $\F f = \what{f}$ for $f\in \F^\inv(C_c^\infty(\calO)) \subset \mathrm{L}^2(\R^d)$, according to Ref. \refcite{Voigtlaender2015PHD} Theorem 3.2.1. Here, $\what{f}$ denotes the usual Fourier transform of a function.

\item The phrase \textit{the associated isomorphic decomposition space} in (3) is justified, because the same reasoning as in the proof of Corollary 3.6.4 in Ref. \refcite{RKDoktorarbeit} shows that this space is uniquely determined.
\end{enumerate}
\end{remark}

We will employ the  following theorem  to obtain sufficient and necessary conditions for the embedding of shearlet coorbit spaces in three dimensions into Sobolev spaces in the sense of the last definition. Here, we define $p^{\triangledown}:=\min\left\{ p,p'\right\}$, where $p'=\infty$ for $0<p <1$ and $p'$ is the usual conjugate exponent for $p\geq 1$.

\begin{theorem}[Ref. \refcite{VoigtlaenderEmbeddingsOfDecSpInSobolebAndBVSPaces} Corollary 5.2]\label{thm:SobolevEmbedding}
Let $\Q=(T_iQ+b_i)_{i\in I}$ be a structured admissible covering of some open set $\mathcal{O}\subset\R^d$. Let
$p,q, r\in (0,\infty],\ k\in \N_0$
and let $u=(u_i)_{i\in I}$ be a $\calQ$-moderate weight on $I$. Define the weight
\begin{align*}
w_i^{\{q\}}:= \abs*{\det\left(T_i\right)}^{\frac{1}{p} - \frac{1}{q}}\left(1+\abs{b_i}^k + \norm*{T_i}^k\right),
\end{align*}
for $i\in I$, where $\norm*{\cdot}$ is some norm on $\GL(\R^d)$. Then the following hold:
\begin{enumerate}[i)]
\item If $p\leq q$ and 
$
\frac{w^{\{q\}}}{u} \in \ell^{q^\triangledown \cdot \left(r/q^\triangledown\right)'}(I),
$
then
$
\calD(\calQ, \mathrm{L}^p, \ell^r_u) \hookrightarrow W^{k,q}(\R^d).
$

\item If $\calD(\calQ, \mathrm{L}^p, \ell^r_u) \hookrightarrow W^{k,q}(\R^d)$, then $p\leq q$ and
$$\frac{w^{\{q\}}}{u} \in \ell^{q \cdot \left(r/q\right)'}(I) \text{ for } q<\infty
\quad\text{ and}\quad
\frac{w^{\{q\}}}{u} \in \ell^{r'}(I) \text{ for } q=\infty.
$$
\end{enumerate}
\end{theorem}

\begin{remark}
Since the necessary and sufficient conditions coincide for $q\in (0,2]\cup\{\infty\}$, this theorem provides a characterization for the embedding 
$$\calD(\calQ, \mathrm{L}^p, \ell^r_u) \hookrightarrow W^{k,q}(\R^d)$$
for $q$ in this range.
\end{remark}

Our computations will make repeated use of the symbol $\asymp$ denoting equivalence between families of scalars. More precisely,  $(a_i)_{i\in I} \asymp (b_i)_{i\in I}$ means there exist constants $C, c>0$ such that $c a_i \leq b_i \leq Ca_i$ for all $i\in I$. We will write this as $a_i \asymp b_i$ if the set $I$ is clear from the context. Moreover, we denote with $\lfloor a \rfloor$ and $\lceil b \rceil$ the biggest integer smaller than $a$ and the smallest integer bigger than $b$, respectively. 

We note the following elementary but useful observations. The proof is omitted. 

\begin{lemma} \label{lem:AsymptoticLemma}
We have the asymptotic relation 
$
\sum_{m=m_0}^\infty m^\rho \asymp m_0^{1+\rho}
$
for $m_0\in \N$ and $\rho < -1$.

If $(a_i)_{i\in I}$ is a family of real numbers with the property that there exists  $\delta \geq 0$ such that $a_i \geq\delta$ for all $i\in I$, then
\begin{enumerate}[i)]
\item $\lfloor a_i\rfloor + 1 \asymp a_i + 1 $,
\item if $\delta >0$, then $\lceil a_i \rceil ^s \asymp a_i^s$ for all $s\in \R$,
\item if $\delta >0$, then $\lfloor a_i \rfloor + 1 \asymp a_i$,
\item if $\delta \geq 1$, then $\lfloor a_i \rfloor^s\asymp a_i^s$ for all $s\in \R$,
\item if $\delta \geq 1$, then $\left(\lfloor a_i \rfloor + 1 \right)\lfloor a_i \rfloor\asymp \lfloor a_i \rfloor^2 \asymp a_i^2$.
\end{enumerate}
\end{lemma}
%

The next lemma is also easily verified.

\begin{lemma}\label{lem:AsymptoticLemma2}
Let $(a_i)_{i\in I}$, $(b_i)_{i\in I}$ and $(c_i)_{i\in I}$ be families of nonnegative numbers with $b_i, c_i \leq a_i$ for all $i\in I$. If there exist sets $I_1, I_2 \subset I$ with $I_1 \cup I_2 = I$, then $a_i \asymp b_i$ on $I_1$ and $a_i \asymp c_i$ on $I_2$ imply $a_i \asymp b_i + c_i$ on $I$ and if $a_i >0$ for all $i\in I$, then  $a_i^s \asymp (b_i + c_i)^s$ for $i\in I$ and all $s\in \R$ 

Furthermore, under the above assumptions, we have
$
\sum_{i\in I} a_i < \infty \Leftrightarrow  \sum_{i\in I} (b_i +c_i) < \infty,$
and if $a_i>0$ for all $i\in I$, then
$
\sum_{i\in I} a_i^s < \infty \Leftrightarrow  \sum_{i\in I} (b_i +c_i)^s < \infty$
for all $s\in \R$.
\end{lemma}

\section{Embeddings of shearlet coorbit spaces into Sobolev spaces}

\label{section:Embeddings_shearlets}

Now that all preliminaries are dealt with, we can take up the task of applying the general results described in the previous sections to the special setup of shearlet dilation groups in dimension three. Following the programme developed above, we now need to use, for each group under consideration, the well spread families in Lemma \ref{lem:StandardWellSpread} and Lemma \ref{lem:ToeplitzWellSpread}, in order to apply Theorem \ref{thm:SobolevEmbedding}. We will treat the case of the standard shearlet groups in more or less full detail. By comparison, our treatment of the Toeplitz shearlet groups is less complete. Here the central estimates, and the way they are obtained, turn out to be very similar to the calculations made for the standard shearlet case, which is why we refrain from including all the details. These can be found in Ref. \refcite{RKDoktorarbeit}.

\subsection{The standard shearlet groups}\label{section:EmbeddingStandardShearlet}

In this subsection, we consider the class of standard shearlet groups $H^{\lambda_1, \lambda_2}$ from section \ref{section:ShearletGroups} and their associated coorbit spaces in dimension three. Our first task is to prepare the application of Theorem \ref{thm:SobolevEmbedding}. 

According to Theorem \ref{thm: FourierIsoCoorbitDecSpaces} and Lemma \ref{lem:StandardWellSpread}, the map
$$
\F:\Co(\mathrm{L}^{p,r}_v(\R^3 \rtimes H^{\lambda_1, \lambda_2} )) \to \mathcal{D}(\mathcal{C}^{\lambda_1, \lambda_2}, \mathrm{L}^p, \ell^r_u)
$$
is an isomorphism for $p,r\in (0,\infty]$ and any weight $v$ on $H^{\lambda_1, \lambda_2}$
that is right moderate with respect to a locally bounded weight on $H^{\lambda_1, \lambda_2}$ if $u$ is a $\mathcal{C}^{\lambda_1, \lambda_2}$-discretization of $v$. Here, we set $I:=\Z^3\times \{\pm 1\}$ and 
\begin{align*}
u_{n,m_1,m_2,\epsilon}&:=v^{(r)}\left(A^{\lambda_1, \lambda_2}_{n, m_1, m_2, \epsilon}\right)
= \abs*{\det\left(\left(A^{\lambda_1, \lambda_2}_{n, m_1, m_2, \epsilon}\right)^\invT\right)}^{\frac{1}{2} - \frac{1}{r}} v\left(\left(A^{\lambda_1, \lambda_2}_{n, m_1, m_2, \epsilon}\right)^\invT\right),
\end{align*}
where the matrices $A^{\lambda_1, \lambda_2}_{n, m_1, m_2, \epsilon}$ were defined in Lemma \ref{lem:StandardWellSpread} and the weight $v^{(r)}$ is defined as in Definition \ref{def:DecompositionWeight}.

As in Ref. \refcite{VoigtlaenderEmbeddingsOfDecSpInSobolebAndBVSPaces}, we will restrict attention to weights 
$
v^{(\alpha, \beta)}: h\mapsto \abs{h_{1,1}}^\alpha \norm*{h^\invT}^\beta,
$
where $h_{1,1}$ denotes the component in the first row and first column of $h$ and $\alpha \in \R, \beta \geq 0$. Informally speaking, this means that $\alpha$ influences how we gauge the scaling factor of the matrix and $\beta$ how we weigh the shearing part. This weight is a generalization of the weight that is considered in Ref. \refcite{DahlkeShearletCoorbitSpacesCompactlySupported} in the two-dimensional case.

For this specific weight, the associated weight $u^{(\alpha, \beta)}$ for the decomposition space is given by
\begin{align*}
u_{n,m_1,m_2,\epsilon}^{(\alpha, \beta)} &= \abs*{\det\left(\left(A^{\lambda_1, \lambda_2}_{n, m_1, m_2, \epsilon}\right)^\invT\right)}^{\frac{1}{2} - \frac{1}{r}} v^{(\alpha, \beta)}\left(\left(A^{\lambda_1, \lambda_2}_{n, m_1, m_2, \epsilon}\right)^\invT\right)\\
& = 2^{-n(1+\lambda_1+ \lambda_2)(\frac{1}{2}- \frac{1}{r})} 2^{-n\alpha} \norm*{A^{\lambda_1, \lambda_2}_{n, m_1, m_2, \epsilon}}^\beta
\end{align*}
according to Definition \ref{def:DecompositionWeight}, where one has to keep in mind that the matrices $A^{\lambda_1, \lambda_2}_{n, m_1, m_2, \epsilon}$ are the inverse transposes of a well spread family.

The weight $w^{\{q\}}$ in Theorem \ref{thm:SobolevEmbedding} is in this setting given by
\begin{align*}
w_{n,m_1,m_2,\epsilon}^{\{q\}}&:= \abs*{\det\left(A^{\lambda_1, \lambda_2}_{n, m_1, m_2, \epsilon}\right)}^{\frac{1}{p} - \frac{1}{q}}\left(1 + \norm*{A^{\lambda_1, \lambda_2}_{n, m_1, m_2, \epsilon}}^k\right)\\
&= 2^{n(1+\lambda_1+ \lambda_2)(\frac{1}{p}- \frac{1}{q})}\left(1 + \norm*{A^{\lambda_1, \lambda_2}_{n, m_1, m_2, \epsilon}}^k\right).
\end{align*}

An application of Theorem \ref{thm:SobolevEmbedding} boils down to the study of the sequence $\zeta^{\lambda_1, \lambda_2}$ defined by
\begin{align*}
\zeta_{n,m_1,m_2,\epsilon}^{\lambda_1, \lambda_2}&:= \frac{w^{\{q\}}_{n,m_1,m_2,\epsilon}}{u^{(\alpha, \beta)}_{n,m_1,m_2,\epsilon}}
=\frac{2^{n(1+\lambda_1+ \lambda_2)(\frac{1}{p}- \frac{1}{q})}\left(1 + \norm*{A^{\lambda_1, \lambda_2}_{n, m_1, m_2, \epsilon}}^k\right)}{2^{-n(1+\lambda_1+ \lambda_2)(\frac{1}{2}- \frac{1}{r})} 2^{-n\alpha} \norm*{A^{\lambda_1, \lambda_2}_{n, m_1, m_2, \epsilon}}^\beta}\\
&= 2^{n\left[\alpha + (1+\lambda_1+\lambda_2)(\frac{1}{2}-\frac{1}{r}+\frac{1}{p}-\frac{1}{q})\right]} \left(\norm*{A^{\lambda_1, \lambda_2}_{n, m_1, m_2, \epsilon}}^{-\beta} + \norm*{A^{\lambda_1, \lambda_2}_{n, m_1, m_2, \epsilon}}^{k-\beta}\right).
\end{align*}

More precisely, we want to characterize $\zeta^{\lambda_1, \lambda_2}\in \ell^\theta(I)$ for
$I=\Z^3\times\{\pm 1\}$
and $\theta\in (0,\infty]$. Since $\zeta_{n,m_1,m_2, 1}^{\lambda_1, \lambda_2} = \zeta_{n,m_1,m_2, -1}^{\lambda_1, \lambda_2}\geq 0$, it is sufficient to characterize $(\zeta_{n,m_1,m_2,1}^{\lambda_1, \lambda_2})_{n,m_1,m_2}\in \ell^\theta(\Z^3)$. Furthermore,
$
\zeta_{n,m_1,m_2,1}^{\lambda_1, \lambda_2} = \psi_{n,m_1,m_2}^{(a, \beta)} + \psi_{n,m_1,m_2}^{(a, \beta-k)} 
$
with $\psi_{n,m_1,m_2}^{(a, \beta)}:= 2^{na}\norm*{A^{\lambda_1, \lambda_2}_{n, m_1, m_2, 1}}^{-\beta}$ and 
$a:=\alpha + (1+\lambda_1+\lambda_2)\left(\frac{1}{2}-\frac{1}{r}+\frac{1}{p}-\frac{1}{q}\right). $

Note that $\psi_{n,m_1,m_2}^{(a, b)}\geq 0$, which implies the equivalence
\begin{align}\label{eq:sequenceStandardShearlet}
\zeta^{\lambda_1, \lambda_2} \in \ell^\theta(I) \quad \Longleftrightarrow \quad \psi^{(a, \beta)},\ \psi^{(a, \beta-k)}\in \ell^\theta(\Z^3).
\end{align}

In summary, our task is reduced to the investigation of the sequence $\psi^{(a, b)}$ for $a,b \in \R$ with $b\geq 0$ and to finding a characterization in terms of $a,b, \theta, \lambda_1, \lambda_2$ for  $\psi^{(a, b)} \in \ell^\theta(\Z^3)$ with 
$a, \lambda_1,\ \lambda_2 \in \R,\ b \geq 0 \text{ and } \theta\in (0,\infty].$
Since all norms on $\GL(\R^3)$ are equivalent, and equivalent norms lead to the same conditions for membership of the considered sequences in $\ell^\theta$-spaces, from now on, we will consider the norm
$\norm*{h}=\sum_{1\leq i,j \leq 3} |h_{i,j}|,$
where $h_{i,j}$ are the components of the matrix $h\in \GL(\R^3)$.
The relation
$\psi^{(a,b)}_{0,m_1,m_2} = \norm*{A^{\lambda_1, \lambda_2}_{0, m_1, m_2, 1}}^{-b} = (3+|m_1|+ |m_2|)^{-b}$
shows that $b\geq 0$ is necessary for $\psi^{(a,b)} \in \ell^\theta(\Z^3) \subset \ell^\infty(\Z^3)$. This is also the reason why we restrict the weights $v^{(\alpha, \beta)}$ to the range of parameters $\alpha\in \R$ and $\beta\geq 0$.

In the following subsections, we will focus on the case $\lambda_1 \leq \lambda_2$ and associated sub-cases. Since 
$\norm*{A^{\lambda_1, \lambda_2}_{n, m_1, m_2, 1}}= 2^n + 2^{n\lambda_1} + 2^{n\lambda_2} + 2^{\lambda_1}|m_1| + 2^{\lambda_2}|m_2|$
is invariant under a change of the index $1$ to $2$ and vice versa, the results for the case $\lambda_2 \leq \lambda_1$ follow by interchanging $\lambda_1$ and $\lambda_2$ in the conditions of the appropriate sub-case.

\subsection{Standard Shearlet Group: case $1\leq \lambda_1 \leq \lambda_2$}\label{subsec:StandardShearletCase}

The general proof strategy consists in determining the asymptotic behavior of $\norm*{A^{\lambda_1, \lambda_2}_{n, m_1, m_2, 1}}$ on suitable subsets of $\Z^3$, and then combining the different conditions on the exponents and weights that arise from the requirement that summation over each subset converges. More precisely, we break $\mathbb{Z}^3$ down into two discrete half spaces and then further into eight octants, and study summation over these subsets. 

\begin{definition}
We define the subsets of $\Z^3$
\begin{align*}
M_1^+&:=\Set{(n,m_1,m_2) | n\geq 0,\ 2^{n\lambda_1}|m_1| \leq  2^{n\lambda_2}|m_2|,\ 2^{n\lambda_2} \leq 2^{n\lambda_2}|m_2|}\\
M_2^+&:=\Set{(n,m_1,m_2) | n\geq 0,\ 2^{n\lambda_1}|m_1| \leq  2^{n\lambda_2}|m_2|,\ 2^{n\lambda_2}|m_2|\leq  2^{n\lambda_2}}\\
M_3^+&:=\Set{(n,m_1,m_2) | n\geq 0,\ 2^{n\lambda_2}|m_2| \leq  2^{n\lambda_1}|m_1|,\ 2^{n\lambda_2} \leq2^{n\lambda_1}|m_1|}\\
M_4^+&:=\Set{(n,m_1,m_2) | n\geq 0,\ 2^{n\lambda_2}|m_2| \leq  2^{n\lambda_1}|m_1|,\ 2^{n\lambda_1}|m_1| \leq 2^{n\lambda_2}}
\shortintertext{and}
M_1^-&:=\Set{(n,m_1,m_2) | n < 0,\ 2^{n\lambda_1}|m_1| \leq  2^{n\lambda_2}|m_2|,\ 2^{n} \leq 2^{n\lambda_2}|m_2|}\\
M_2^-&:=\Set{(n,m_1,m_2) | n <0 ,\ 2^{n\lambda_1}|m_1| \leq  2^{n\lambda_2}|m_2|,\ 2^{n\lambda_2}|m_2|\leq  2^{n}}\\
M_3^-&:=\Set{(n,m_1,m_2) | n < 0,\ 2^{n\lambda_2}|m_2| \leq  2^{n\lambda_1}|m_1|,\ 2^{n} \leq2^{n\lambda_1}|m_1|}\\
M_4^-&:=\Set{(n,m_1,m_2) | n < 0,\ 2^{n\lambda_2}|m_2| \leq  2^{n\lambda_1}|m_1|,\ 2^{n\lambda_1}|m_1| \leq 2^{n}}.
\end{align*}
Furthermore, we let $M^{\dagger} = \bigcup_{i=1,\ldots,4} M_i^\dagger$, for $\dagger \in \{ +,- \}$. 
\end{definition}

Note that the union of all sets in the above definition is $\Z^3$. We introduce these sets because we want to use 
$\psi^{(a, b)}\in \ell^\theta(\Z^3) \Longleftrightarrow \psi^{(a, b)}\in \ell^\theta(M^{\dagger}_i)\text{ for all } \dagger \in \{\pm\}, i\in \{1,\ldots, 4\}.$

\begin{lemma}\label{lem:AsymptoticStandard1}
The following asymptotic relations hold for $(n,m_1,m_2)$ in the given sets:
\begin{align*}
\norm*{A^{\lambda_1, \lambda_2}_{n, m_1, m_2, 1}} \asymp 
\begin{cases}
2^{n\lambda_2}|m_2|,& \text{for } (n,m_1,m_2) \in M_1^+\\ 
2^{n\lambda_2}, 	&\text{for } (n,m_1,m_2) \in M_2^+\\
2^{n\lambda_1}|m_1|,&\text{for } (n,m_1,m_2) \in M_3^+\\
2^{n\lambda_2},		&\text{for } (n,m_1,m_2) \in M_4^+\\
2^{n\lambda_2}|m_2|,&\text{for } (n,m_1,m_2) \in M_1^-\\
2^{n},				&\text{for } (n,m_1,m_2) \in M_2^-\\
2^{n\lambda_1}|m_1|,&\text{for } (n,m_1,m_2) \in M_3^-\\
2^{n},				&\text{for } (n,m_1,m_2) \in M_4^-
\end{cases}
\end{align*}
\end{lemma}

\begin{proof}
We have
$\norm*{A^{\lambda_1, \lambda_2}_{n, m_1, m_2, 1}}= 2^n + 2^{n\lambda_1} + 2^{n\lambda_2} + 2^{n\lambda_1}|m_1| + 2^{n\lambda_2}|m_2|.$
Hence, if $n\geq0$, the inequality 
$2^{n\lambda_2} \leq 2^n + 2^{n\lambda_1} + 2^{n\lambda_2} \leq 3\cdot 2^{n\lambda_2} $
shows that $$\norm*{A^{\lambda_1, \lambda_2}_{n, m_1, m_2, 1}}\asymp 2^{n\lambda_2} + 2^{\lambda_1}|m_1| + 2^{\lambda_2}|m_2|$$ holds and for $n<0$, the inequality $2^{n} \leq 2^n + 2^{n\lambda_1} + 2^{n\lambda_2} \leq 3\cdot 2^{n}$ implies $\norm*{A^{\lambda_1, \lambda_2}_{n, m_1, m_2, 1}}\asymp 2^{n} + 2^{\lambda_1}|m_1| + 2^{\lambda_2}|m_2|$. The rest follows by observing that the sets  $M_i^\dagger$ for $i\in \{1,\ldots, 4\}$ and $\dagger\in\{\pm\}$ are precisely chosen in a way to ensure that one term dominates the other terms in these asymptotic norm expressions.
\end{proof}

The next step consists in using these asymptotic relations in order to characterize $\psi^{(a, b)}\in \ell^\theta(\Z^3)$. We first consider $\theta=\infty$.

\begin{lemma}\label{lem:ThetaInfty1Lambda1Lambda2}
For $1\leq \lambda_1 \leq \lambda_2$, $a\in \R$ and $b\geq 0$, we have
$$
\psi^{(a, b)}\in \ell^\infty(\Z^3) \Longleftrightarrow b \leq a \leq b\lambda_2.
$$
\end{lemma}

\begin{proof}
We will determine successively necessary and sufficient conditions for   $\psi^{(a, b)}\in \ell^\infty(M_i^\dagger)$ for $i\in \{1,\ldots, 4\}$ and $\dagger\in\{\pm\}$.
\begin{enumerate}[i)]
\item The set $M_1^+$: In this case, we have
\begin{align*}
&\norm*{\psi^{(a, b)}}_{\ell^\infty(M^+_1)} =\sup_{(n,m_1,m_2) \in M^+_1 } 2^{na}\norm*{A^{\lambda_1, \lambda_2}_{n,m_1,m_2,1}}^{-b} 
	 \asymp \sup_{(n,m_1,m_2) \in M^+_1 } 2^{na} 2^{-bn\lambda_2}|m_2|^{-b}\\
    & = \sup_{n\geq 0}
    		\sup_{\substack{2^{n\lambda_1}|m_1| \leq 2^{n\lambda_2}|m_2| \\ 1 \leq |m_2|}} 2^{n(a-b\lambda_2)}|m_2|^{-b}
    \stackrel{(\ast)}{=} \sup_{n\geq 0} 2^{n(a-b\lambda_2)} < \infty
\end{align*}
if and only if $a-b\lambda_2 \leq 0$. In $(\ast)$ we used that the expression $|m_2|^{-b}$ achieves its maximum with respect to $(m_1,m_2)$ with the given restriction for the choice $(m_1, m_2) = (0,1)$.

\item The set $M_2^+$: In this case, we have
\begin{align*}
&\norm*{\psi^{(a, b)}}_{\ell^\infty(M^+_2)} =\sup_{(n,m_1,m_2) \in M^+_2 } 2^{na}\norm*{A^{\lambda_1, \lambda_2}_{n,m_1,m_2,1}}^{-b} 
	 \asymp \sup_{(n,m_1,m_2) \in M^+_2 } 2^{na} 2^{-nb\lambda_2}\\
    & = \sup_{n\geq 0}\sup_{\substack{2^{n\lambda_1}|m_1| \leq 2^{n\lambda_2}|m_2| \\ 2^{n\lambda_2}|m_2| \leq 2^{n\lambda_2} }} 2^{n(a-b\lambda_2)}
    \stackrel{(\ast)}{=} \sup_{n\geq 0} 2^{n(a-b\lambda_2)} < \infty
\end{align*}
if and only if $a-b\lambda_2 \leq 0$. In $(\ast)$ we used that the expression we take the supremum of is independent of $(m_1, m_2)$.

\item The set $M_3^+$:  In this case, we have
\begin{align*}
&\norm*{\psi^{(a, b)}}_{\ell^\infty(M^+_3)} =\sup_{(n,m_1,m_2) \in M^+_3 } 2^{na}\norm*{A^{\lambda_1, \lambda_2}_{n,m_1,m_2,1}}^{-b} 
	\asymp \sup_{(n,m_1,m_2) \in M^+_3 } 2^{na} 2^{-bn\lambda_1}|m_1|^{-b}\\
    & = \sup_{n\geq 0}
    		\sup_{\substack{2^{n\lambda_2}|m_2| \leq 2^{n\lambda_1}|m_1| \\ 2^{n(\lambda_2-\lambda_1)} \leq |m_1|}} 2^{n(a-b\lambda_1)}|m_1|^{-b}
     \stackrel{(\ast)}{=} \sup_{n\geq 0} 2^{n(a-b\lambda_1)} \lceil 2^{n(\lambda_2-\lambda_1)} \rceil^{-b}\\
    & \stackrel{\ref{lem:AsymptoticLemma}}{\asymp} \sup_{n\geq 0} 2^{n(a-b\lambda_1)}2^{-bn(\lambda_2-\lambda_1)} =   \sup_{n\geq 0}2^{n(a-b\lambda_2)} < \infty
\end{align*}
if and only if $a-b\lambda_2 \leq 0$. In $(\ast)$ we used that the expression $|m_1|^{-b}$ achieves its maximum with respect to $(m_1,m_2)$ with the given restriction for the choice $(m_1, m_2) = (\lceil 2^{n(\lambda_2-\lambda_1)} \rceil,0 )$.

\item The set $M_4^+$: In this case, we have
\begin{align*}
&\norm*{\psi^{(a, b)}}_{\ell^\infty(M^+_4)} =\sup_{(n,m_1,m_2) \in M^+_4 } 2^{na}\norm*{A^{\lambda_1, \lambda_2}_{n,m_1,m_2,1}}^{-b} 
	\asymp \sup_{(n,m_1,m_2) \in M^+_4 } 2^{na} 2^{-nb\lambda_2}\\
    & = \sup_{n\geq 0}\sup_{\substack{2^{n\lambda_2}|m_2| \leq 2^{n\lambda_1}|m_1| \\ 2^{n\lambda_1}|m_1| \leq 2^{n\lambda_2} }} 2^{n(a-b\lambda_2)}
    \stackrel{(\ast)}{=} \sup_{n\geq 0} 2^{n(a-b\lambda_2)} < \infty
\end{align*}
if and only if $a-b\lambda_2 \leq 0$. In $(\ast)$ we used that the expression we take the supremum of is independent of $(m_1, m_2)$.

\item The set $M^+$: The previous four cases and Lemma \ref{lem:AsymptoticLemma2} combined with Lemma \ref{lem:AsymptoticStandard1} lead to
\begin{align*}
\norm*{\psi^{(a, b)}}_{\ell^\infty(M^+)}
\asymp \sup_{(n,m_1,m_2) \in M^+}\abs*{2^{na} \left(2^{n\lambda_2} + 2^{n\lambda_1}|m_1| + 2^{n\lambda_2}|m_2|\right)^{-b}} < \infty
\end{align*}
if and only if $a-b\lambda_2 \leq 0$. 

\item The set $M_1^-$: In this case, we have
\begin{align*}
&\norm*{\psi^{(a, b)}}_{\ell^\infty(M^-_1)} =\sup_{(n,m_1,m_2) \in M^-_1 } 2^{na}\norm*{A^{\lambda_1, \lambda_2}_{n,m_1,m_2,1}}^{-b} 
	\asymp \sup_{(n,m_1,m_2) \in M^-_1 } 2^{na} 2^{-bn\lambda_2}|m_2|^{-b}\\
    & = \sup_{n< 0}
    		\sup_{\substack{2^{n\lambda_1}|m_1| \leq 2^{n\lambda_2}|m_2| \\ 2^n \leq 2^{n\lambda_2}|m_2| }} 2^{n(a-b\lambda_2)}|m_2|^{-b}
    \stackrel{(\ast)}{=}\sup_{n< 0}2^{n(a-b\lambda_2)}\lceil 2^{n(1-\lambda_2)}\rceil^{-b} \\
    & \stackrel{\ref{lem:AsymptoticLemma}}{\asymp} \sup_{n< 0} 2^{n(a-b\lambda_2)} 2^{-bn(1-\lambda_2)} =\sup_{n<0}2^{n(a-b)} < \infty
\end{align*}
if and only if $a-b \geq 0$. In $(\ast)$ we used that the expression $|m_2|^{-b}$ achieves its maximum with respect to $(m_1,m_2)$ with the given restriction for the choice $(m_1, m_2) = (0,\lceil 2^{n(1-\lambda_2)}\rceil)$.

\item The set $M_2^-$: In this case, we have
\begin{align*}
&\norm*{\psi^{(a, b)}}_{\ell^\infty(M^-_2)} =\sup_{(n,m_1,m_2) \in M^-_2 } 2^{na}\norm*{A^{\lambda_1, \lambda_2}_{n,m_1,m_2,1}}^{-b} 
	\asymp \sup_{(n,m_1,m_2) \in M^-_2 } 2^{na} 2^{-nb}\\
    & = \sup_{n\geq 0}\sup_{\substack{2^{n\lambda_1}|m_1| \leq 2^{n\lambda_2}|m_2| \\ 2^{n\lambda_2}|m_2| \leq 2^{n} }} 2^{n(a-b)}
    \stackrel{(\ast)}{=} \sup_{n\geq 0} 2^{n(a-b)} < \infty
\end{align*}
if and only if $a-b \geq 0$. In $(\ast)$ we used that the expression we take the supremum of is independent of $(m_1, m_2)$.

\item The set $M_3^-$: In this case, we have
\begin{align*}
&\norm*{\psi^{(a, b)}}_{\ell^\infty(M^-_3)} =\sup_{(n,m_1,m_2) \in M^-_3 } 2^{na}\norm*{A^{\lambda_1, \lambda_2}_{n,m_1,m_2,1}}^{-b} 
	 \asymp \sup_{(n,m_1,m_2) \in M^-_3 } 2^{na} 2^{-bn\lambda_1}|m_1|^{-b}\\
    & = \sup_{n< 0}
    		\sup_{\substack{2^{n\lambda_2}|m_2| \leq 2^{n\lambda_1}|m_1| \\ 2^n \leq 2^{n\lambda_1}|m_1| }} 2^{n(a-b\lambda_1)}|m_1|^{-b}
     \stackrel{(\ast)}{=}\sup_{n< 0}2^{n(a-b\lambda_1)}\lceil 2^{n(1-\lambda_1)}\rceil^{-b}\\
   &\stackrel{\ref{lem:AsymptoticLemma}}{\asymp} \sup_{n< 0} 2^{n(a-b\lambda_1)} 2^{-bn(1-\lambda_1)} =2^{n(a-b)} < \infty
\end{align*}
if and only if $a-b \geq 0$. In $(\ast)$ we used that the expression $|m_1|^{-b}$ achieves its maximum with respect to $(m_1,m_2)$ with the given restriction for the choice $(m_1, m_2) = (\lceil 2^{n(1-\lambda_1)}\rceil, 0)$.

\item The set $M_4^-$: In this case, we have
\begin{align*}
&\norm*{\psi^{(a, b)}}_{\ell^\infty(M^-_4)} =\sup_{(n,m_1,m_2) \in M^-_4} 2^{na}\norm*{A^{\lambda_1, \lambda_2}_{n,m_1,m_2,1}}^{-b} 
	 \asymp \sup_{(n,m_1,m_2) \in M^-_4} 2^{na} 2^{-nb}\\
    & = \sup_{n\geq 0}\sup_{\substack{2^{n\lambda_2}|m_2| \leq 2^{n\lambda_1}|m_1|  2^{n\lambda_1}|m_1| \leq 2^{n} }} 2^{n(a-b)}
    \stackrel{(\ast)}{=} \sup_{n\geq 0} 2^{n(a-b)} < \infty
\end{align*}
if and only if $a-b \geq 0$. In $(\ast)$ we used that the expression we take the supremum of is independent of $(m_1, m_2)$.

\item The set $M^-$: The previous four cases and Lemma \ref{lem:AsymptoticLemma2} combined with Lemma \ref{lem:AsymptoticStandard1} lead to
\begin{align*}
&\norm*{\psi^{(a, b)}}_{\ell^\infty(M^-)}
\asymp \sup_{(n,m_1,m_2) \in M^-}\abs*{2^{na} \left(2^{n} + 2^{n\lambda_1}|m_1| + 2^{n\lambda_2}|m_2|\right)^{-b}} < \infty
\end{align*}
if and only if $a-b\geq 0$.  

\item The set $\Z^3$: To summarize our results, we have 
$\psi^{(a, b)}\in \ell^\infty(\Z^3)  < \infty$
if and only if $b \leq a \leq b\lambda_2$.
\end{enumerate}
\end{proof}

The next step is to consider the remaining $\theta<\infty$.

\begin{lemma}\label{lem:ThetaNotInfty1Lambda1Lambda2}
For $1\leq \lambda_1 \leq \lambda_2$, $a\in \R$ and $b\geq 0$, and $\theta \in (0,\infty)$, we have 
$\psi^{(a, b)}\in \ell^\theta(\Z^3)$
if and only if $b\theta > 2$ and 
$
\frac{\lambda_1 + \lambda_2 -2}{\theta} + b < a < \frac{\lambda_1 - \lambda_2}{\theta} + b\lambda_2.
$
\end{lemma}

\begin{proof}
We proceed as in the previous proof:
\begin{enumerate}[i)]
\item The set $M_1^+$: In this case, we have
\begin{align*}
&\norm*{\psi^{(a, b)}}_{\ell^\theta(M^+_1)}^\theta = \sum_{(n,m_1,m_2) \in M^+_1}\abs*{2^{na}\norm*{A^{\lambda_1, \lambda_2}_{n,m_1,m_2,1}}^{-b}}^\theta
	 \asymp  \sum_{(n,m_1,m_2) \in M^+_1} \left(2^{na} 2^{-bn\lambda_2}|m_2|^{-b}\right)^\theta\\
    & = \sum_{n\geq 0}\sum_{\substack{1 \leq |m_2| \\ 2^{n\lambda_1} |m_1| \leq 2^{n\lambda_2} |m_2|}} 2^{\theta n(a-b\lambda_2)}|m_2|^{-b\theta}
     \asymp \sum_{n= 0}^\infty\sum_{m_2=1}^\infty \sum_{m_1=0}^{\lfloor 2^{n(\lambda_2 - \lambda_1)} m_2 \rfloor} 2^{\theta n(a-b\lambda_2)} m_2^{-b\theta}\\
    & = \sum_{n= 0}^\infty\sum_{m_2=1}^\infty \left(\lfloor 2^{n(\lambda_2 - \lambda_1)} m_2 \rfloor + 1\right) 2^{\theta n(a-b\lambda_2)} m_2^{-b\theta}
     \stackrel{(\ref{lem:AsymptoticLemma})}{\asymp} \sum_{n= 0}^\infty\sum_{m_2=1}^\infty 2^{n(\lambda_2 - \lambda_1)} m_2 2^{\theta n(a-b\lambda_2)} m_2^{-b\theta}\\
     & = \left(\sum_{n= 0}^\infty  2^{n( \theta( a-b\lambda_2) + \lambda_2 - \lambda_1)}\right) \left(\sum_{m_2=1}^\infty m_2^{1-b\theta}\right) < \infty\\
\end{align*}
if and only if $b\theta>2$ and $\theta( a-b\lambda_2) + \lambda_2 - \lambda_1<0$.

\item The set $M_2^+$: We can decompose the set $M_2^+$ as 
$M_2^+ = M_2^{+{''}} \cup M_2^{+{'}},$
where
\begin{align*}
 M_2^{+{'}}&:= M_2^{+} \cap \Set{(n,m_1, 0) | n,m_1\in \Z}\quad\text{ and}\\
 M_2^{+{''}}&:= M_2^{+} \cap \Set{(n,m_1, m_2) | n,m_1, m_2\in \Z \text{ and } |m_2|=1}.
\end{align*}
Since $M_2^{+{''}} \subset M_1^+$, the conditions $b\theta>2$ and 
$\theta( a-b\lambda_2) + \lambda_2 - \lambda_1<0$
ensure also $\norm*{\psi^{(a, b)}}_{\ell^\theta( M_2^{+{''}})} < \infty$. For the set  $M_2^{+{'}}$, we get
\begin{align*}
&\norm*{\psi^{(a, b)}}_{\ell^\theta(M_2^{+{'}})}^\theta = \sum_{(n,m_1,m_2) \in M_2^{+{'}}}\abs*{2^{na}\norm*{A^{\lambda_1, \lambda_2}_{n,m_1,m_2,1}}^{-b}}^\theta
	 \stackrel{\ref{lem:AsymptoticStandard1}}{\asymp}  \sum_{(n,m_1,m_2) \in M_2^{+{'}}} \left(2^{na} 2^{-bn\lambda_2}\right)^\theta\\
    & = \sum_{n\geq 0}\sum_{\substack{m_2 = 0 \\ m_1 = 0}} 2^{\theta n(a-b\lambda_2)}
    = \sum_{n= 0}^\infty 2^{\theta n(a-b\lambda_2)} < \infty
\end{align*}

if and only if $\theta( a-b\lambda_2)<0$, which is weaker condition than $\theta( a-b\lambda_2) + \lambda_2 - \lambda_1<0$ since $\lambda_2 - \lambda_1 \geq 0$.

\item The set $M_3^+$: In this case, we have for $b\theta > 2$ the following additional condition
\begin{align*}
&\norm*{\psi^{(a, b)}}_{\ell^\theta(M^+_3)}^\theta = \sum_{(n,m_1,m_2) \in M^+_3}\abs*{2^{na}\norm*{A^{\lambda_1, \lambda_2}_{n,m_1,m_2,1}}^{-b}}^\theta
	 \asymp  \sum_{(n,m_1,m_2) \in M^+_3} \left(2^{na} 2^{-bn\lambda_1}|m_1|^{-b}\right)^\theta\\
    & = \sum_{n\geq 0}\sum_{\substack{ 2^{n\lambda_2}|m_2| \leq 2^{n\lambda_1}|m_1| \\ 2^{n\lambda_2} \leq 2^{n\lambda_1} |m_1|}} 2^{\theta n(a-b\lambda_1)}|m_1|^{-b\theta}
    \asymp \sum_{n= 0}^\infty\sum_{m_1= \lceil 2^{n(\lambda_2 - \lambda_1)}\rceil}^\infty \sum_{m_2=0}^{\lfloor 2^{n(\lambda_1 - \lambda_2)} m_1 \rfloor} 2^{\theta n(a-b\lambda_1)} m_1^{-b\theta}\\
 	 & = \sum_{n= 0}^\infty\sum_{m_1= \lceil 2^{n(\lambda_2 - \lambda_1)}\rceil}^\infty  \left(\lfloor 2^{n(\lambda_1 - \lambda_2)} m_1 \rfloor + 1 \right)2^{\theta n(a-b\lambda_1)} m_1^{-b\theta}\\
     &\stackrel{\ref{lem:AsymptoticLemma}}{\asymp} \sum_{n= 0}^\infty\sum_{m_1= \lceil 2^{n(\lambda_2 - \lambda_1)}\rceil}^\infty   2^{n(\lambda_1 - \lambda_2)} m_1 2^{\theta n(a-b\lambda_1)} m_1^{-b\theta}\\
     &=\sum_{n= 0}^\infty\sum_{m_1= \lceil 2^{n(\lambda_2 - \lambda_1)}\rceil}^\infty 2^{n(\theta (a-b\lambda_1) +\lambda_1 - \lambda_2) } m_1^{1-b\theta}
     \stackrel{\ref{lem:AsymptoticLemma}}{\asymp} \sum_{n= 0}^\infty 2^{n(\theta (a-b\lambda_1) +\lambda_1 - \lambda_2) }  \lceil 2^{n(\lambda_2 - \lambda_1)}\rceil ^{2-b\theta} \\
      & \stackrel{\ref{lem:AsymptoticLemma}}{\asymp} \sum_{n= 0}^\infty 2^{n(\theta (a-b\lambda_1) +\lambda_1 - \lambda_2) } 2^{n(\lambda_2 - \lambda_1)(2-b\theta)} 
      = \sum_{n= 0}^\infty 2^{n( \theta(a-b\lambda_2)+ \lambda_2 - \lambda_1 )}< \infty
\end{align*}
if and only if $\theta( a-b\lambda_2) + \lambda_2 - \lambda_1<0$.

\item The set $M_4^+$: In this case, we have
\begin{align*}
&\norm*{\psi^{(a, b)}}_{\ell^\theta(M^+_4)}^\theta = \sum_{(n,m_1,m_2) \in M^+_4}\abs*{2^{na}\norm*{A^{\lambda_1, \lambda_2}_{n,m_1,m_2,1}}^{-b}}^\theta
	\asymp  \sum_{(n,m_1,m_2) \in M^+_4} \left(2^{na} 2^{-bn\lambda_2}\right)^\theta\\
    & = \sum_{n\geq 0}\sum_{\substack{ 2^{n\lambda_1} |m_1| \leq 2^{n\lambda_2} \\ 2^{n\lambda_2} |m_2| \leq 2^{n\lambda_1} |m_1|}} 2^{\theta n(a-b\lambda_2)}
    \asymp \sum_{n= 0}^\infty\sum_{m_1=0}^{\lfloor 2^{n(\lambda_2 - \lambda_1)}\rfloor} \sum_{m_2=0}^{\lfloor 2^{n(\lambda_1 - \lambda_2)} m_1 \rfloor} 2^{\theta n(a-b\lambda_2)}\\
    & \stackrel{(\ast)}{\asymp}\sum_{n= 0}^\infty\sum_{m_1=0}^{\lfloor 2^{n(\lambda_2 - \lambda_1)}\rfloor} 2^{\theta n(a-b\lambda_2)}
    =\sum_{n = 0}^\infty\left(\lfloor 2^{n(\lambda_2 - \lambda_1)}\rfloor + 1\right) 2^{\theta n(a-b\lambda_2)}\\
    & \stackrel{\ref{lem:AsymptoticLemma}}{\asymp} \sum_{n = 0}^\infty  2^{n(\lambda_2 - \lambda_1)}2^{\theta n(a-b\lambda_2)}
    =\sum_{n = 0}^\infty  2^{n(\theta(a-b\lambda_2) +\lambda_2 - \lambda_1)} < \infty
    \end{align*}
if and only if $\theta( a-b\lambda_2) + \lambda_2 - \lambda_1<0$. In $(\ast)$ we used that for $0\leq m_1 \leq \lfloor 2^{n(\lambda_2 - \lambda_1)} \rfloor$ the inequality 
$0\leq m_2 \leq \lfloor 2^{n(\lambda_1 - \lambda_2)} m_1 \rfloor$
implies $m_2\in \{0, 1\}$, which means that the value of the sum we omitted in this step is in $\{1,2\}$.

\item The set $M^+$: The previous four cases and Lemma \ref{lem:AsymptoticLemma2} combined with Lemma \ref{lem:AsymptoticStandard1} lead to
\begin{align*}
\norm*{\psi^{(a, b)}}_{\ell^\theta(M^+)}^\theta
\asymp \sum_{(n,m_1,m_2) \in M^+}\abs*{2^{na} \left(2^{n\lambda_2} + 2^{n\lambda_1}|m_1| + 2^{n\lambda_2}|m_2|\right)^{-b}}^\theta < \infty
\end{align*}
if and only if $\theta( a-b\lambda_2) + \lambda_2 - \lambda_1<0$ and $b\theta> 2$ for the set $M^+:=\cup_{i=1}^4M_i^+$.

\item The set $M_1^-$: In this case, we have under the already established condition $b\theta > 2$
\begin{align*}
&\norm*{\psi^{(a, b)}}_{\ell^\theta(M^-_1)}^\theta = \sum_{(n,m_1,m_2) \in M^-_1}\abs*{2^{na}\norm*{A^{\lambda_1, \lambda_2}_{n,m_1,m_2,1}}^{-b}}^\theta
	 \asymp  \sum_{(n,m_1,m_2) \in M^-_1} \left(2^{na} 2^{-bn\lambda_2}|m_2|^{-b}\right)^\theta\\
    & = \sum_{n < 0}\sum_{\substack{2^n \leq  2^{n\lambda_2} |m_2| \\ 2^{n\lambda_1} |m_1| \leq 2^{n\lambda_2} |m_2|}} 2^{\theta n(a-b\lambda_2)}|m_2|^{-b\theta}
     \asymp \sum_{n< 0}^\infty\sum_{m_2= \lceil 2^{n(1-\lambda_2)} \rceil}^\infty \sum_{m_1=0}^{\lfloor 2^{n(\lambda_2 - \lambda_1)} m_2 \rfloor} 2^{\theta n(a-b\lambda_2)} m_2^{-b\theta}\\
    &= \sum_{n< 0}^\infty\sum_{m_2= \lceil 2^{n(1-\lambda_2)} \rceil}^\infty  \left( \lfloor 2^{n(\lambda_2 - \lambda_1)} m_2 \rfloor + 1\right) 2^{\theta n(a-b\lambda_2)} m_2^{-b\theta}\\
     &\stackrel{\ref{lem:AsymptoticLemma}}{\asymp}\sum_{n< 0}^\infty\sum_{m_2= \lceil 2^{n(1-\lambda_2)} \rceil}^\infty 2^{n(\lambda_2 - \lambda_1)} m_2 2^{\theta n(a-b\lambda_2)} m_2^{-b\theta}\\
     & =\sum_{n< 0}^\infty\sum_{m_2= \lceil 2^{n(1-\lambda_2)} \rceil}^\infty 2^{ n(\theta(a-b\lambda_2)+\lambda_2 - \lambda_1)} m_2^{1-b\theta}
      \stackrel{\ref{lem:AsymptoticLemma}}{\asymp}\sum_{n< 0}^\infty2^{ n(\theta(a-b\lambda_2)+\lambda_2 - \lambda_1)}  \lceil 2^{n(1-\lambda_2)} \rceil^{2-b\theta} \\
     & \stackrel{\ref{lem:AsymptoticLemma}}{\asymp}\sum_{n< 0}2^{ n(\theta(a-b\lambda_2)+\lambda_2 - \lambda_1)}2^{n(1-\lambda_2)(2-b\theta)}
      = \sum_{n< 0}2^{ n(\theta(a-b)+ 2 -\lambda_2 - \lambda_1)} <\infty
\end{align*}
if and only if  and $\theta( a-b) + 2 - \lambda_2 - \lambda_1>0$.

\item The set $M_2^-$: In this case, we have
\begin{align*}
&\norm*{\psi^{(a, b)}}_{\ell^\theta(M^-_2)}^\theta = \sum_{(n,m_1,m_2) \in M^-_2}\abs*{2^{na}\norm*{A^{\lambda_1, \lambda_2}_{n,m_1,m_2,1}}^{-b}}^\theta
	 \asymp  \sum_{(n,m_1,m_2) \in M^-_2} \left(2^{na} 2^{-bn}\right)^\theta\\
    & = \sum_{n < 0}\sum_{\substack{2^{n\lambda_1}|m_1| \leq  2^{n\lambda_2} |m_2| \\ 2^{n\lambda_2} |m_2| \leq 2^{n}}} 2^{\theta n(a-b)}
    \asymp \sum_{n < 0}\sum_{m_2=0}^{\lfloor 2^{n(1-\lambda_2)}\rfloor}\sum_{m_1=0}^{\lfloor 2^{n(\lambda_2-\lambda_1)}m_2\rfloor}2^{\theta n(a-b)}\\
    &=\sum_{n < 0}\sum_{m_2=0}^{\lfloor 2^{n(1-\lambda_2)}\rfloor} \left( \lfloor 2^{n(\lambda_2-\lambda_1)}m_2\rfloor + 1\right)2^{\theta n(a-b)}
   \stackrel{\ref{lem:AsymptoticLemma}}{\asymp} \sum_{n < 0}\sum_{m_2=0}^{\lfloor 2^{n(1-\lambda_2)}\rfloor}  \left(2^{n(\lambda_2-\lambda_1)}m_2+1\right) 2^{\theta n(a-b)}.
\end{align*}

We expand the sum and consider the two following series separately, first

\begin{align*}
&\sum_{n < 0}\sum_{m_2=0}^{\lfloor 2^{n(1-\lambda_2)}\rfloor}  2^{n(\theta(a-b) + \lambda_2-\lambda_1)}m_2 
\stackrel{(\ast)}{\asymp} \sum_{n < 0}2^{n(\theta(a-b) + \lambda_2-\lambda_1)}\lfloor 2^{n(1-\lambda_2)}\rfloor\left(\lfloor 2^{n(1-\lambda_2)}\rfloor + 1\right)\\
& =\sum_{n < 0}2^{n(\theta(a-b) + \lambda_2-\lambda_1)}\lfloor 2^{n(1-\lambda_2)}\rfloor^2  + \sum_{n < 0}2^{n(\theta(a-b) + \lambda_2-\lambda_1)}\lfloor 2^{n(1-\lambda_2)}\rfloor \\
& \stackrel{(\ref{lem:AsymptoticLemma})}{\asymp} \sum_{n < 0}2^{n(\theta(a-b) + \lambda_2-\lambda_1)}2^{2n(1-\lambda_2)}+ \sum_{n < 0}2^{n(\theta(a-b) + \lambda_2-\lambda_1)} 2^{n(1-\lambda_2)}\\
& = \sum_{n < 0} 2^{n(\theta(a-b)+ 2 - \lambda_1 - \lambda_2)} + \sum_{n < 0} 2^{n(\theta(a-b)+1 - \lambda_1 )} < \infty
\end{align*}
if and only if $\theta( a-b) + 2 - \lambda_2 - \lambda_1>0$ and 
$\theta(a-b)+1 - \lambda_1 > 0,$
where the first inequality implies the second inequality because $\lambda_2 - 1 \geq 0$ and in $(\ast)$ we used $\sum_{m=1}^n m = \frac{1}{2}n(n+1).$

Now, we consider the remaining series
\begin{align*}
\sum_{n < 0}\sum_{m_2=0}^{\lfloor 2^{n(1-\lambda_2)}\rfloor} 2^{n\theta(a-b)}= \sum_{n < 0}2^{n\theta(a-b)} \left(\lfloor 2^{n(1-\lambda_2)}\rfloor + 1 \right)
&\stackrel{\ref{lem:AsymptoticLemma}}{\asymp} \sum_{n < 0}2^{n\theta(a-b)} 2^{n(1-\lambda_2)}\\
 & = \sum_{n < 0}2^{n(\theta(a-b) + 1 - \lambda_2 )}< \infty 
\end{align*}
if and only if $\theta(a-b)+1 - \lambda_2 > 0$, which is implied by the already established inequality $\theta( a-b) + 2 - \lambda_2 - \lambda_1>0$ since $\lambda_1 - 1 \geq 0$.

In summary, we have $\norm*{\psi^{(a, b)}}_{\ell^\theta(M^-_2)}^\theta < \infty \Longleftrightarrow \theta( a-b) + 2 - \lambda_2 - \lambda_1>0.$

\item The set $M_3^-$: In this case, we have under the already established condition $b\theta > 2$
\begin{align*}
&\norm*{\psi^{(a, b)}}_{\ell^\theta(M^-_3)}^\theta = \sum_{(n,m_1,m_2) \in M^-_3}\abs*{2^{na}\norm*{A^{\lambda_1, \lambda_2}_{n,m_1,m_2,1}}^{-b}}^\theta
	\asymp  \sum_{(n,m_1,m_2) \in M^-_3} \left(2^{na} 2^{-bn\lambda_1}|m_1|^{-b}\right)^\theta\\
    & = \sum_{n < 0}\sum_{\substack{2^{n\lambda_2} |m_2| \leq  2^{n\lambda_1} |m_1| \\ 2^{n} \leq 2^{n\lambda_1} |m_1|}} 2^{\theta n(a-b\lambda_1)}|m_1|^{-b\theta}
	 \asymp \sum_{n < 0}\sum_{m_1=\lceil 2^{n(1-\lambda_1)} \rceil}^\infty \sum_{m_2=0}^{\lfloor 2^{n(\lambda_1 - \lambda_2)} m_1\rfloor} 2^{\theta n(a-b\lambda_1)}m_1^{-b\theta}\\
    & = \sum_{n < 0}\sum_{m_1=\lceil 2^{n(1-\lambda_1)} \rceil}^\infty \left(\lfloor 2^{n(\lambda_1 - \lambda_2)} m_1\rfloor + 1\right)2^{\theta n(a-b\lambda_1)}m_1^{-b\theta}\\
    &\stackrel{\ref{lem:AsymptoticLemma}}{\asymp} \sum_{n < 0}\sum_{m_1=\lceil 2^{n(1-\lambda_1)} \rceil}^\infty 2^{n(\lambda_1 - \lambda_2)} m_1 2^{\theta n(a-b\lambda_1)}m_1^{-b\theta}\\
    &=\sum_{n < 0}\sum_{m_1=\lceil 2^{n(1-\lambda_1)} \rceil}^\infty 2^{n(\theta(a-b\lambda_1)+\lambda_1 - \lambda_2)}m_1^{1-b\theta}
    \stackrel{\ref{lem:AsymptoticLemma}}{\asymp} \sum_{n < 0}2^{n(\theta(a-b\lambda_1)+\lambda_1 - \lambda_2)}\lceil 2^{n(1-\lambda_1)} \rceil^{2-b\theta}\\
    &\stackrel{\ref{lem:AsymptoticLemma}}{\asymp} \sum_{n < 0}2^{n(\theta(a-b\lambda_1)+\lambda_1 - \lambda_2)} 2^{n(1-\lambda_1)(2-b\theta)}
    = \sum_{n < 0}2^{n(\theta(a-b)+ 2 -\lambda_1 - \lambda_2)}< \infty
\end{align*}
if and only if $\theta( a-b) + 2 - \lambda_2 - \lambda_1>0$.

\item The set $M_4^-$: In this case, we have 
\begin{align*}
&\norm*{\psi^{(a, b)}}_{\ell^\theta(M^-_4)}^\theta = \sum_{(n,m_1,m_2) \in M^-_4}\abs*{2^{na}\norm*{A^{\lambda_1, \lambda_2}_{n,m_1,m_2,1}}^{-b}}^\theta
	 \asymp  \sum_{(n,m_1,m_2) \in M^-_4} \left(2^{na} 2^{-bn}\right)^\theta\\
    & = \sum_{n < 0}\sum_{\substack{2^{n\lambda_2}|m_2| \leq  2^{n\lambda_1} |m_1| \\ 2^{n\lambda_1} |m_1| \leq 2^{n}}} 2^{\theta n(a-b)} < \infty
     \asymp  \sum_{n < 0}\sum_{m_1=0}^{\lfloor 2^{n(1-\lambda_1)}\rfloor}\sum_{m_2=0}^{\lfloor 2^{n(\lambda_1-\lambda_2)}m_1\rfloor}2^{\theta n(a-b)}\\
    & = \sum_{n < 0}\sum_{m_1=0}^{\lfloor 2^{n(1-\lambda_1)}\rfloor} \left( \lfloor 2^{n(\lambda_1-\lambda_2)}m_1\rfloor + 1\right) 2^{\theta n(a-b)}
    \stackrel{\ref{lem:AsymptoticLemma}}{\asymp}\sum_{n < 0}\sum_{m_1=0}^{\lfloor 2^{n(1-\lambda_1)}\rfloor} \left(2^{n(\lambda_1-\lambda_2)}m_1 + 1\right) 2^{\theta n(a-b)}
\end{align*}
and this sum is finite if and only if the following two sums are finite. The first sum is
\begin{align*}
 &\sum_{n < 0}\sum_{m_1=0}^{\lfloor 2^{n(1-\lambda_1)}\rfloor} 2^{n(\theta(a-b)+\lambda_1-\lambda_2)}m_1
 \stackrel{(\ast)}{\asymp}\sum_{n < 0} 2^{n(\theta(a-b)+\lambda_1-\lambda_2)}\lfloor 2^{n(1-\lambda_1)}\rfloor\left(\lfloor 2^{n(1-\lambda_1)}\rfloor + 1\right)\\
& \stackrel{\ref{lem:AsymptoticLemma}}{\asymp} \sum_{n < 0} 2^{n(\theta(a-b)+\lambda_1-\lambda_2)} 2^{2n(1-\lambda_1)}
 = \sum_{n < 0} 2^{n(\theta(a-b)+2-\lambda_1-\lambda_2)} < \infty
\end{align*}
if and only if  $\theta( a-b) + 2 - \lambda_2 - \lambda_1>0$, where we used in $(\ast)$ again 
$\sum_{m=1}^n m = \frac{1}{2}n(n+1)$. The second sum is
\begin{align*}
 \sum_{n < 0}\sum_{m_1=0}^{\lfloor 2^{n(1-\lambda_1)}\rfloor} 2^{n\theta(a-b)}
 = \sum_{n < 0}2^{n\theta(a-b)}\left(\lfloor 2^{n(1-\lambda_1)}\rfloor + 1 \right)
& \stackrel{\ref{lem:AsymptoticLemma}}{\asymp} \sum_{n < 0}2^{n\theta(a-b)} 2^{n(1-\lambda_1)}\\
&=  \sum_{n < 0}2^{n(\theta(a-b)+1-\lambda_1)} < \infty 
\end{align*}
if and only if $\theta(a-b)+1-\lambda_1>0$, which is a weaker condition than for the first sum because $\lambda_2 \geq 1$.

\item The set $M^-$: The previous four cases and Lemma \ref{lem:AsymptoticLemma2} combined with Lemma \ref{lem:AsymptoticStandard1} lead to
\begin{align*}
&\norm*{\psi^{(a, b)}}_{\ell^\theta(M^-)}^\theta
\asymp \sum_{(n,m_1,m_2) \in M^-}\abs*{2^{na} \left(2^{n} + 2^{n\lambda_1}|m_1| + 2^{n\lambda_2}|m_2|\right)^{-b}}^\theta < \infty
\end{align*}
if and only if $\theta( a-b) + 2 - \lambda_2 - \lambda_1>0$ and $b\theta> 2$ for the set $M^-:=\cup_{i=1}^4M_i^-$.

\item The set $\Z^3$: To summarize our results, we have 
$\psi^{(a, b)}\in \ell^\theta(M^- \cup M^+) = \ell^\theta(\Z^3)  < \infty$
if and only if $b\theta> 2$ and
\begin{align*}
\theta( a-b\lambda_2) + \lambda_2 - \lambda_1&	<\quad 0			<\theta( a-b) + 2 - \lambda_2 - \lambda_1\\
\Longleftrightarrow \frac{\lambda_1 + \lambda_2 -2}{\theta} + b &	<\quad a 		< \frac{\lambda_1 - \lambda_2}{\theta} + b\lambda_2.
\end{align*}
\end{enumerate}
\end{proof}

If we put the last results together, we get a complete characterization for these $1 \leq \lambda_1 \leq  \lambda_2$.

\begin{corollary}\label{cor:ResultatStandard1}
The following are equivalent to  $\psi^{(a, b)}\in \ell^\theta(\Z^3)$  for $1\leq \lambda_1 \leq \lambda_2$, $a\in \R$ and $b\geq 0$: 
\begin{align*}
\begin{cases}
b\theta >2,\ \frac{\lambda_1 + \lambda_2 -2}{\theta} + b < a 		< \frac{\lambda_1 - \lambda_2}{\theta} + b\lambda_2, & \text{if } \theta \in (0,\infty) \\
b\leq a \leq b\lambda_2, & \text{if } \theta = \infty.
\end{cases}
\end{align*}
\end{corollary}

\subsection{Standard Shearlet Group: remaining cases}

By proceeding in a completely analogous manner in the cases $\lambda_1 \leq \lambda_2 \leq 1$ and $\lambda_1 \leq 1 \leq \lambda_2 $, we achieve for these remaining cases similar results, which we summarize in the next corollary.

\begin{corollary}\label{cor:StanardEmbeddingPrep}
The following are equivalent to  $\psi^{(a, b)}\in \ell^\theta(\Z^3)$ for $a\in \R$ and $b\geq 0$:
\begin{enumerate}[i)]

\item If $1\leq \lambda_1 \leq \lambda_2$:
\begin{align*}
\begin{cases}
b\theta >2,\ \frac{\lambda_1 + \lambda_2 -2}{\theta} + b < a 		< \frac{\lambda_1 - \lambda_2}{\theta} + b\lambda_2, & \text{if } \theta \in (0,\infty) \\
b\leq a \leq b\lambda_2, & \text{if } \theta = \infty.
\end{cases}
\end{align*}

\item If $\lambda_1 \leq \lambda_2 \leq 1$:
\begin{align*}
\begin{cases}
b\theta >2,\ \frac{\lambda_2 - \lambda_1}{\theta} + b\lambda_1 	< a < \frac{\lambda_1 + \lambda_2 - 2}{\theta} + b, & \text{if } \theta \in (0,\infty) \\
b\lambda_1\leq a \leq b, & \text{if } \theta = \infty.
\end{cases}
\end{align*}

\item If $\lambda_1 \leq 1 \leq \lambda_2 $:
\begin{align*}
\begin{cases}
b\theta >2,\ \frac{\lambda_2 - \lambda_1}{\theta} + b\lambda_1 	< a < \frac{\lambda_1 - \lambda_2}{\theta} + b\lambda_2, & \text{if } \theta \in (0,\infty) \\
b\lambda_1\leq a \leq b\lambda_2, & \text{if } \theta = \infty.
\end{cases}
\end{align*}
And the case for $\lambda_2 \leq \lambda_1$ corresponds to the respective case for $\lambda_1 \leq \lambda_2$ by interchanging $\lambda_1$ and $\lambda_2$.
\end{enumerate}
\end{corollary}

Recall that we were actually interested in conditions that characterize the finiteness of the norm of $\zeta_{n,m_1,m_2,1}^{\lambda_1, \lambda_2} = \psi_{n,m_1,m_2}^{(a, \beta)} + \psi_{n,m_1,m_2}^{(a, \beta-k)}$ in order to apply Theorem \ref{thm:SobolevEmbedding} to decide and characterize the existence of an embedding $\Co(L^{p,r}_m(\R^d\rtimes H^{\lambda_1, \lambda_2}))\hookrightarrow W^{k,q}(\R^d)$ in the sense of Definition \ref{def:EmbeddingSobolev}.

\begin{theorem}\label{Cor:EmbeddingStandardLong}
Let $p, q, r\in (0,\infty]$, $k\in \N_0$ and set 
$\gamma:=\left(\frac{1}{2}-\frac{1}{r}+\frac{1}{p}-\frac{1}{q}\right).$
The embedding $\Co(\mathrm{L}^{p,r}_{v^{(\alpha, \beta)}}(\R^3\rtimes H^{\lambda_1, \lambda_2}))\hookrightarrow W^{k,q}(\R^3)$ holds for 
$q \in (0,2]\cup \{\infty\},$
if and only if $p\leq q$ and
\begin{enumerate}[i)]
\item for $r\leq q^\triangledown$
	\begin{enumerate}
	\item the inequalities
    \begin{align*}
    \beta \leq \alpha + (1+\lambda_1 + \lambda_2)\gamma \leq (\beta-k)\lambda_2
    \end{align*}
    hold if $1\leq \lambda_1 \leq \lambda_2$,
    
	\item the inequalities
    \begin{align*}
    \max\{\beta \lambda_1, (\beta-k) \lambda_1\} \leq \alpha + (1+\lambda_1 + \lambda_2)\gamma \leq \beta-k
    \end{align*}
    hold if $\lambda_1 \leq \lambda_2 \leq 1$,
    
    \item   the inequalities
    \begin{align*}
    \max\{\beta \lambda_1, (\beta-k) \lambda_1\} \leq \alpha + (1+\lambda_1 + \lambda_2)\gamma \leq (\beta-k)\lambda_2
    \end{align*}
    hold if $\lambda_1 \leq 1 \leq \lambda_2$.
	\end{enumerate}
    
\item for $r> q^\triangledown$ the inequality $\beta > k +2\left(\frac{1}{q^\triangledown} - \frac{1}{r}\right)$ and
	\begin{enumerate}
	\item the inequalities
    \begin{align*}
    (\lambda_1 + \lambda_2 - 2)\left(\frac{1}{q^\triangledown} - \frac{1}{r}\right) + \beta 
    <  \alpha + (1+\lambda_1 + \lambda_2)\gamma
    < (\lambda_1-\lambda_2)\left(\frac{1}{q^\triangledown} - \frac{1}{r}\right) + (\beta-k)\lambda_2
    \end{align*}
    hold if $1\leq \lambda_1 \leq \lambda_2$,
    
	\item the inequalities
    \begin{align*}
    (\lambda_2 - \lambda_1)\left(\frac{1}{q^\triangledown} - \frac{1}{r}\right) + \max\{\beta \lambda_1, (\beta-k) \lambda_1\}
    &<  \alpha + (1+\lambda_1 + \lambda_2)\gamma \\
    &<(\lambda_1 + \lambda_2 -2)\left(\frac{1}{q^\triangledown} - \frac{1}{r}\right) + (\beta-k)
    \end{align*}
    hold if $\lambda_1 \leq \lambda_2\leq 1$,
    
	\item the inequalities
    \begin{align*}
            (\lambda_2 - \lambda_1)\left(\frac{1}{q^\triangledown} - \frac{1}{r}\right) + \max\{\beta \lambda_1, (\beta-k) \lambda_1\}
    &<  \alpha + (1+\lambda_1 + \lambda_2)\gamma \\
    &< (\lambda_1-\lambda_2)\left(\frac{1}{q^\triangledown} - \frac{1}{r}\right) + (\beta-k)\lambda_2
    \end{align*}
    hold if $\lambda_1 \leq 1 \leq \lambda_2$.
	\end{enumerate}
\end{enumerate}    
For the remaining case $q\in (2,\infty)$, the given conditions are sufficient for the embedding $\Co(\mathrm{L}^{p,r}_{v^{(\alpha, \beta)}}(\R^3\rtimes H^{\lambda_1, \lambda_2})\hookrightarrow W^{k,q}(\R^3)$, and necessary for this embedding if one replaces $q^\triangledown$ with $q$ in the inequalities.
The case for $\lambda_2 \leq \lambda_1$ corresponds to the respective case for $\lambda_1 \leq \lambda_2$ by interchanging $\lambda_1$ and $\lambda_2$.

\end{theorem}

\begin{proof}
Essentially, this is an application of Theorem \ref{thm:SobolevEmbedding} and the prior discussion in Section \ref{section:EmbeddingStandardShearlet}. Since
$\zeta^{\lambda_1, \lambda_2} \in \ell^\theta(\Z^3\times\{\pm 1\}) \quad \Longleftrightarrow \quad \psi^{(a, \beta)},\ \psi^{(a, \beta-k)}\in \ell^\theta(\Z^3)$
with $a=\alpha + (1+\lambda_1 + \lambda_2)\gamma$. And the condition 
$\zeta^{\lambda_1, \lambda_2} \in \ell^{q^\triangledown \left(\frac{r}{q^\triangledown}\right)'}(\Z^3)$
is sufficient for the existence of the embedding 
$\Co(\mathrm{L}^{p,r}_{v^{(\alpha, \beta)}}(\R^3\rtimes H^{\lambda_1, \lambda_2}))\hookrightarrow W^{k,q}(\R^3)$ for $q\in (2,\infty)$ and equivalent to it for $q \in (0,2]\cup \{\infty\}$.

We just have to set $\theta := q^\triangledown \left(\frac{r}{q^\triangledown}\right)'$ and notice that
$q^\triangledown \left(\frac{r}{q^\triangledown}\right)' = \infty  \Longleftrightarrow r\leq q^\triangledown$
and that in the case $r> q^\triangledown$, we get
$q^\triangledown \left(\frac{r}{q^\triangledown}\right)' = q^\triangledown \frac{\frac{r}{q^\triangledown}}{\frac{r}{q^\triangledown}-1} = \frac{1}{\frac{1}{q^\triangledown} - \frac{1}{r}}.$
Hence $\theta^\inv= \frac{1}{q^\triangledown} - \frac{1}{r}$, where all these computations also hold for $q=\infty$ or $r=\infty$ or $p=\infty$ with $\frac{1}{\infty}=0$. The application of Corollary \ref{cor:StanardEmbeddingPrep} with these parameters then leads to the given conditions in the given cases.
\end{proof}

\subsection{Embeddings for Toeplitz shearlet groups}\label{section:EmbeddingToeplitzShearlet}

In this section, we consider the class of Toeplitz shearlet groups $H^{\delta}$ from section \ref{section:ShearletGroups} and their associated coorbit spaces in dimension three. Again, according to Theorem \ref{thm: FourierIsoCoorbitDecSpaces} and Lemma \ref{lem:ToeplitzWellSpread}, the map
$\F:\Co(L^{p,r}_v(\R^3 \rtimes H^{\delta} )) \to \mathcal{D}(\mathcal{C}^{\delta}, L^p, \ell^r_u)$
is an isomorphism for $p,r\in (0,\infty]$ and any weight 
$v:H^{\delta}\to (0,\infty)$ that is right moderate with respect to a locally bounded weight on $H^{\delta}$ if $u:I\to(0,\infty)$ is a $\mathcal{C}^{\delta}$-discretization of a decomposition weight for the coorbit space $\Co(L^{p,r}_v(\R^3 \rtimes H^{\delta}))$. The associated weight has the following form 
$
u_{n,m_1,m_2,\epsilon}:=v^{(r)}\left(A^{\delta}_{n, m_1, m_2, \epsilon}\right)
 = \abs*{\det\left(\left(A^{\delta}_{n, m_1, m_2, \epsilon}\right)^\invT\right)}^{\frac{1}{2} - \frac{1}{r}} v\left(\left(A^{\delta}_{n, m_1, m_2, \epsilon}\right)^\invT\right),
$
where the matrices $A^{\delta}_{n, m_1, m_2, \epsilon}$ were defined in Lemma \ref{lem:ToeplitzWellSpread}.

As in the standard shearlet case, we restrict attention to weights 
$
v^{(\alpha, \beta)}: h\mapsto \abs{h_{1,1}}^\alpha \norm*{h^\invT}^\beta,
$
where $\alpha \in \R$ and $\beta \geq 0$. For this specific weight, the associated weight $u^{(\alpha, \beta)}$ for the decomposition space is given by
\begin{align*}
u_{n,m_1,m_2,\epsilon}^{(\alpha, \beta)} &= \abs*{\det\left(\left(A^{\delta}_{n, m_1, m_2, \epsilon}\right)^\invT\right)}^{\frac{1}{2} - \frac{1}{r}} v^{(\alpha, \beta)}\left(\left(A^{\delta}_{n, m_1, m_2, \epsilon}\right)^\invT\right)\\
& = 2^{-3n(1-\delta)(\frac{1}{2}- \frac{1}{r})} 2^{-n\alpha} \norm*{A^{\delta}_{n, m_1, m_2, \epsilon}}^\beta,
\end{align*}
according to Definition \ref{def:DecompositionWeight}, where one has to keep in mind that the matrices $A^{\delta}_{n, m_1, m_2, \epsilon}$ are the inverse transposes of a well spread family.

The weight $w^{\{q\}}$ in Theorem \ref{thm:SobolevEmbedding} is in this setting given by
\begin{align*}
w_{n,m_1,m_2,\epsilon}^{\{q\}}&:= \abs*{\det\left(A^{\delta}_{n, m_1, m_2, \epsilon}\right)}^{\frac{1}{p} - \frac{1}{q}}\left(1 + \norm*{A^{\delta}_{n, m_1, m_2, \epsilon}}^k\right)\\
&= 2^{3n(1-\delta)(\frac{1}{p}- \frac{1}{q})}\left(1 + \norm*{A^{\delta}_{n, m_1, m_2, \epsilon}}^k\right).
\end{align*}

An application of Theorem \ref{thm:SobolevEmbedding} boils down to the study of the sequence $\zeta^{\delta}$ defined by
\begin{align*}
\zeta_{n,m_1,m_2,\epsilon}^{\delta}&:= \frac{w^{(q)}_{n,m_1,m_2,\epsilon}}{u^{(\alpha, \beta)}_{n,m_1,m_2,\epsilon}}
=\frac{2^{3n(1-\delta)(\frac{1}{p}- \frac{1}{q})}\left(1 + \norm*{A^{\delta}_{n, m_1, m_2, \epsilon}}^k\right)}{2^{-3n(1 - \delta)(\frac{1}{2}- \frac{1}{r})} 2^{-n\alpha} \norm*{A^{\delta}_{n, m_1, m_2, \epsilon}}^\beta}\\
&= 2^{n\left[\alpha + 3(1-\delta)(\frac{1}{2}-\frac{1}{r}+\frac{1}{p}-\frac{1}{q})\right]} \left(\norm*{A^{\delta}_{n, m_1, m_2, \epsilon}}^{-\beta} + \norm*{A^{\delta}_{n, m_1, m_2, \epsilon}}^{k-\beta}\right).
\end{align*}

More precisely, we want to characterize $\zeta^{\delta}\in \ell^\theta(I)$ for 
$I=\Z^3\times\{\pm 1\}$
and $\theta\in (0,\infty]$. Since $\zeta_{n,m_1,m_2, 1}^{\delta} = \zeta_{n,m_1,m_2, -1}^{\delta}\geq 0$, it is sufficient to characterize $(\zeta_{n,m_1,m_2,1}^{\delta})_{n,m_1,m_2}\in \ell^\theta(\Z^3)$. Furthermore, $
\zeta_{n,m_1,m_2,1}^{\delta} = \psi_{n,m_1,m_2}^{(a, \beta)} + \psi_{n,m_1,m_2}^{(a, \beta-k)}$
with $\psi_{n,m_1,m_2}^{(a, \beta)}:= 2^{na}\norm*{A^{\delta}_{n, m_1, m_2, 1}}^{-\beta}$ and 
$a:=\alpha + 3(1-\delta)\left(\frac{1}{2}-\frac{1}{r}+\frac{1}{p}-\frac{1}{q}\right).$

Note that $\psi_{n,m_1,m_2}^{(a, b)}\geq 0$, which implies the equivalence
\begin{align*}
\zeta^{\delta} \in \ell^\theta(I) \quad \Longleftrightarrow \quad \psi^{(a, \beta)},\ \psi^{(a, \beta-k)}\in \ell^\theta(\Z^3).
\end{align*}

Hence, our task is again reduced to the investigation of the sequence $\psi^{(a, b)}$ for $a,b \in \R$, and finding a characterization in terms of $a,b, \theta, \delta$ for  $\psi^{(a, b)} \in \ell^\theta(\Z^3)$ with $a,\delta  \in \R, b \geq 0$ and $\theta\in (0,\infty]$. We consider again the norm
$
\norm{h}=\sum_{1\leq i,j \leq 3} |h_{i,j}|.
$
For the same reason as in the standard shearlet case, we restrict the weights $v^{\alpha, \beta}$ to the range of parameters $\alpha\in \R$ and $\beta\geq 0$.

The details of this calculation will not be presented here. Somewhat 
surprisingly, it turns out that no additional, extensive computations of the type presented in Subsection \ref{subsec:StandardShearletCase} have to be performed. The results for the Toeplitz case can be traced back to already established estimates for the standard shearlet case and the choice $\lambda_1 := 1-\delta, \lambda_2:= 1-2\delta$ or variations of this choice. Despite the different shearing behavior of these two classes of groups, the process of applying the embedding results is very similar. For details, we refer to Ref. \refcite{RKDoktorarbeit}. 

\begin{corollary}\label{cor:ResultatToeplitz1}
The following are equivalent to  $\psi^{(a, b)}\in \ell^\theta(\Z^3)$ for $a\in \R$ and $b\geq 0$: if $\delta \geq 0$ 
\begin{align*}
\begin{cases}
b\theta >2,\ b(1-2\delta) < a < -\frac{3\delta}{\theta} + b, & \text{if } \theta \in (0,\infty) \\
b(1-2\delta) \leq a \leq b, & \text{if } \theta = \infty,
\end{cases}
\end{align*}
if $\delta < 0$
\begin{align*}
\begin{cases}
b\theta >2,\ -\frac{3\delta}{\theta} + b < a <  b(1-2\delta), & \text{if } \theta \in (0,\infty) \\
b \leq a \leq b(1-2\delta), & \text{if } \theta = \infty.
\end{cases}
\end{align*}
\end{corollary}

These conditions can again be used to get an embedding result.

\begin{theorem}\label{Cor:EmbeddingToeplitzLong}
Let $p, q, r \in (0,\infty]$, $\alpha \in \R$, $\beta \geq 0$, $k\in \N_0$ and set 
$$\gamma:=\left(\frac{1}{2}-\frac{1}{r}+\frac{1}{p}-\frac{1}{q}\right).$$
The embedding $\Co(\mathrm{L}^{p,r}_{v^{(\alpha, \beta)}}(\R^3\rtimes H^{\delta}))\hookrightarrow W^{k,q}(\R^3)$ holds for 
$q \in (0,2]\cup \{\infty\},$
if and only if $p\leq q$ and
\begin{enumerate}[i)]
\item for $r\leq q^\triangledown$
	\begin{enumerate}
	\item the inequalities
    \begin{align*}
    \max\{\beta(1-2\delta), (\beta- k)(1-2\delta)\}\leq \alpha + 3(1-\delta)\gamma \leq \beta-k
    \end{align*}
    hold if $\delta \geq 0$,
    
	\item the inequalities
    \begin{align*}
    \beta \leq \alpha + 3(1-\delta)\gamma \leq (\beta-k)(1-2\delta)
    \end{align*}
    hold if $\delta < 0$,
	\end{enumerate}
    
\item for $r> q^\triangledown$ the inequality $\beta > k +2\left(\frac{1}{q^\triangledown} - \frac{1}{r}\right)$ and
	\begin{enumerate}
	\item the inequalities
   \begin{align*}
           \max\{\beta(1-2\delta), (\beta-k)(1-2\delta)\}
    <  \alpha + 3(1-\delta)\gamma
    < -3\delta\left(\frac{1}{q^\triangledown} - \frac{1}{r}\right) + \beta - k
   \end{align*}

    hold if $\delta \geq 0$,
    
	\item the inequalities
    \begin{align*}
    -3\delta\left(\frac{1}{q^\triangledown} - \frac{1}{r}\right) + \beta
    <  \alpha + 3(1-\delta)\gamma
    < (\beta-k)(1-2\delta)
    \end{align*}
    hold if $\delta < 0$.
	\end{enumerate}
\end{enumerate}

For the remaining case $q\in (2,\infty)$, the given conditions are sufficient for the embedding $\Co(\mathrm{L}^{p,r}_{v^{(\alpha, \beta)}}(\R^3\rtimes H^{\delta})\hookrightarrow W^{k,q}(\R^3)$, and necessary for this embedding if one replaces $q^\triangledown$ with $q$ in the inequalities.
\end{theorem}

\section{Interpretation of embedding results. Coorbit spaces as smoothness spaces}
\label{Consequences}

In this section, we give a sample application of the results attained in Subsection \ref{section:EmbeddingStandardShearlet} and Section \ref{section:EmbeddingToeplitzShearlet} and compare the embedding behavior of the spaces  
$$\Co(\mathrm{L}^{p}_{v^{(\alpha, \beta)}}(\R^3\rtimes H^{\delta}))\quad \text{and}\quad \Co(\mathrm{L}^{p}_{v^{(\alpha, \beta)}}(\R^3\rtimes H^{\lambda_1, \lambda_2}))$$
for $q\in (0,2]$. We choose $p=r$ in this section to illustrate some consequences of these embedding results in a particularly transparent case.

To make precise what we mean by the embedding behavior of different groups, we make the following definition. 

\begin{definition}
We say that two shearlet dilation groups 
$H_1, H_2 \subset \GL(\R^3)$
have \textit{the same embedding behavior} if for all 
$q\in (0,2],\ p\in (0,\infty],\ \alpha, \beta \in \R$
with $\beta\geq 0$ and $k\in \N_0$, the following equivalence holds
\begin{align*}
\Co(\mathrm{L}^p_{v^{(\alpha, \beta)}}(\R^3\rtimes H_1)) \hookrightarrow W^{k,q}(\R^3)
\Longleftrightarrow \Co(\mathrm{L}^p_{v^{(\alpha, \beta)}}(\R^3\rtimes H_2)) \hookrightarrow W^{k,q}(\R^3).
\end{align*}
If this does not hold, we say $H_1, H_2$ have \textit{different embedding behavior}.
\end{definition}

\begin{remark} \label{rem:embedding_1}
\begin{enumerate}[i)]
\item We restrict this definition to $q\in (0,2]$ because the inequalities in Theorem \ref{Cor:EmbeddingStandardLong} and Theorem \ref{Cor:EmbeddingToeplitzLong} provide sufficient and necessary conditions  for the existence of such an embedding. It is worth noting again that $p$ has no influence on the existence of such an embedding for $q$ in this range, except that the condition $p\leq q$ has to be fulfilled.

\item The relation of having the same embedding behavior is clearly an equivalence relation on the set of shearlet dilation groups in three dimensions.
\end{enumerate}
\end{remark}

As simple example for the kind of result we are interested in establishing in this section, we show that there does not hold an embedding into Sobolev spaces with at least one weak derivative for constant weights (i.e. $\alpha=\beta=0$).    

\begin{corollary} \label{cor:no_embedding}
Let $p, q \in (0,2]$, $k\in \N_0$ be arbitrary. Necessary for the embedding 
$$\Co(\mathrm{L}^{p}_{v^{(\alpha, \beta)}}(\R^3\rtimes H^{\lambda_1, \lambda_2}))\hookrightarrow W^{k,q}(\R^3)$$
is $\beta\geq k$.
\end{corollary}

\begin{proof}
This follows from the earlier observation regarding the sequence $\zeta^{\lambda_1, \lambda_2}$ after the equivalence in (\ref{eq:sequenceStandardShearlet}).  
\end{proof}


The above Corollary has the consequence that in the unweighted case $v\equiv 1$, there are not any smooth functions in the space $\Co(L^{p}(\R^3\rtimes H^{\lambda_1, \lambda_2}))$ if we restrict the search for smooth functions to $W^{k,q}(\R^3)$ for $p, q \in (0,2]$ and $k\geq 1$. This is somewhat remarkable since smaller integrability exponents guarantee a sharper decay condition on the coefficients in the atomic decomposition of $f\in \Co(L^{p}(\R^3\rtimes H^{\lambda_1, \lambda_2}))$ (see Ref. \refcite{Voigtlaender2015PHD} Theorem 2.4.19) and this leads in many instances to higher regularity properties of the functions in the space. However, for some parameters, the intuition that smaller $p$ leads to higher regularity is indeed true as we see in Remark \ref{rem:SmallerSmoother}.

\subsection{Characterizing groups by their embedding behavior}

In this subsection, we clarify when coorbit spaces associated to two different shearlet dilation groups have the same embedding behavior. This question is interesting for three reasons: It allows to determine which components of shearlet groups are truly relevant for embedding theorems; as it turns out, the shearing parts are not relevant (see the remark below). Secondly, these results contrast nicely to the fact, established in Ref. \refcite{RKDoktorarbeit}, that for any pair of distinct shearlet groups in dimension three, the induced scales of coorbit spaces do differ. Hence the embedding behavior into Sobolev spaces does not generally allow to distinguish between shearlet groups. Thirdly, understanding when two shearlet groups induce the same embedding behavior simplifies the subsequent further discussion, by allowing to concentrate on certain subcases.

We first observe that we can actually restrict the question of different embedding behaviors to the class of standard shearlet groups. 

\begin{corollary}\label{Cor:EmbeddingBehavior1}
The groups $H^{\lambda_1, \lambda_2}$ and $H^{\lambda_2, \lambda_1}$ have the same embedding behavior and the groups $H^{\delta}$ and $H^{1-\delta, 1-2\delta}$ have the same embedding behavior for $\lambda_1, \lambda_2, \delta \in \R$.
\end{corollary}

\begin{proof}
This follows immediately by comparing the inequalities given in Theorem \ref{Cor:EmbeddingStandardLong} and Theorem \ref{Cor:EmbeddingToeplitzLong}.
\end{proof}

\begin{remark}
\begin{enumerate}[i)]
\item To phrase this differently, the embedding behavior of shearlet dilation groups in dimension three into Sobolev spaces in our setting is completely determined by the diagonal entries of the group elements. The different shearing parts of Toeplitz and standard shearlet groups have no effect on the existence of embeddings.

\item This also means that we can restrict our attention to the group $H^{\lambda_1, \lambda_2}$ with $\lambda_1 \leq \lambda_2$, when we discuss the embedding behaviour of shearlet coorbit spaces in more detail. The next result shows that within this smaller class, different groups have different embedding behaviors.
\end{enumerate}
\end{remark}

\begin{theorem}\label{Cor:SameEmbeddingStandard}
For $\lambda_1, \lambda_2, \lambda_1', \lambda_2' \in \R$, the groups $H^{\lambda_1, \lambda_2}$ and $H^{\lambda_1', \lambda_2'}$ have the same embedding behavior if and only if $\{\lambda_1, \lambda_2\}=\{\lambda_1', \lambda_2'\}$.
\end{theorem}

\begin{proof} We already know that the embedding behavior is the same if $\{\lambda_1, \lambda_2\}=\{\lambda_1', \lambda_2'\}$, according to Corollary \ref{Cor:EmbeddingBehavior1} and that it suffices to consider the case $\lambda_1 \leq \lambda_2$ and $\lambda_1' \leq \lambda_2'$. Let $p\in (0,\infty]$, $q\in (0,2]$, $\alpha\in \R$, $\beta \geq 0$ and $k\in \N_0$ and $\gamma = 1/2-1/q$ in this proof. According to Theorem \ref{Cor:EmbeddingStandardLong}, the embedding
\begin{align*}
\Co(\mathrm{L}^p_{v^{(\alpha, \beta)}}(\R^3\rtimes H^{\lambda_1, \lambda_2})) \hookrightarrow W^{k,q}(\R^3)
\end{align*}
holds if and only if $p\leq q$ and
\begin{align*}
\beta - (1+\lambda_1 + \lambda_2)\gamma \leq \alpha \leq (\beta-k)\lambda_2 -  (1+\lambda_1 + \lambda_2)\gamma
\end{align*}
for $1\leq \lambda_1 \leq \lambda_2$, denote this inequality with $I_1(\lambda_1, \lambda_2)$ and
\begin{align*}
    \max\{\beta \lambda_1, (\beta-k)\lambda_1\} - (1+\lambda_1 + \lambda_2)\gamma \leq \alpha
\leq (\beta-k) -  (1+\lambda_1 + \lambda_2)\gamma
\end{align*}

for $\lambda_1 \leq \lambda_2 \leq 1$, denote this inequality with $I_2(\lambda_1, \lambda_2)$ and
\begin{align*}
    \max\{\beta \lambda_1, (\beta-k)\lambda_1\} - (1+\lambda_1 + \lambda_2)\gamma \leq \alpha
\leq (\beta-k)\lambda_2 -  (1+\lambda_1 + \lambda_2)\gamma
\end{align*}

for $\lambda_1 \leq 1 \leq \lambda_2 $, denote this inequality with $I_3(\lambda_1, \lambda_2)$.

Assume that the groups $H^{\lambda_1, \lambda_2}$ and $H^{\lambda_1', \lambda_2'}$ have the same embedding behavior and let always be $p\leq q$.

\begin{enumerate}[i)]
\item If $1\leq \lambda_1 \leq \lambda_2$ and $1\leq \lambda_1' \leq \lambda_2'$, this entails that the inequalities $I_1(\lambda_1, \lambda_2)$ and $I_1(\lambda_1', \lambda_2')$ are equivalent for all choices of parameters $p\in (0,\infty]$, $q\in (0,2]$, $\alpha\in \R$, $\beta \geq 0$ and $k\in \N_0$. This implies the equalities
$
\beta - (1+\lambda_1 + \lambda_2)\gamma  = \beta - (1+\lambda_1' + \lambda_2')\gamma
$
and
$
(\beta-k)\lambda_2 -  (1+\lambda_1 + \lambda_2)\gamma = (\beta-k)\lambda_2' -  (1+\lambda_1' + \lambda_2')\gamma.
$
for all $p\in (0,\infty]$, $q\in (0,2]$, $\alpha\in \R$, $\beta \geq 0$ and $k\in \N_0$ (the lower and upper bounds of the inequalities for $\alpha$ have to coincide). The second equality implies for $q=2$ ($\gamma = 0$) and $\beta-k=1 \neq 0$ the equality $\lambda_2 = \lambda_2'$. Then the first equality implies for some $q\neq 2$ ($\gamma \neq 0$) the equality $\lambda_1 = \lambda_1'$.

\item If $1\leq \lambda_1 \leq \lambda_2$ and $\lambda_1' \leq \lambda_2' \leq 1$, this entails that the inequalities $I_1(\lambda_1, \lambda_2)$ and $I_2(\lambda_1', \lambda_2')$ are equivalent for all choices of parameters $p\in (0,\infty]$, $q\in (0,2]\cup$, $\alpha\in \R$, $\beta \geq 0$ and $k\in \N_0$. This implies the equalities
$
\beta - (1+\lambda_1 + \lambda_2)\gamma  = \max\{\beta \lambda_1', (\beta-k)\lambda_1'\} - (1+\lambda_1' + \lambda_2')\gamma
$and$
(\beta-k)\lambda_2 -  (1+\lambda_1 + \lambda_2)\gamma = (\beta-k) -  (1+\lambda_1' + \lambda_2')\gamma.
$
for all $p\in (0,\infty]$, $q\in (0,2]$, $\alpha\in \R$, $\beta \geq 0$ and $k\in \N_0$ (the lower and upper bounds of the inequalities for $\alpha$ have to coincide). The second equality implies for $q=2$ ($\gamma = 0$) and $\beta-k=1 \neq 0$ the equality $\lambda_2 = 1 $. The first equality implies for $q=2$ ($\gamma=0$) and some $\beta >0$ first the inequality $\lambda_1' \geq 0$ and then $\lambda_1'=1$. But this implies already 
$
\lambda_1 = \lambda_1' = \lambda_2 = \lambda_2' =1.
$

\item If $1\leq \lambda_1 \leq \lambda_2$ and $\lambda_1' \leq 1 \leq \lambda_2' $, this entails that the inequalities $I_1(\lambda_1, \lambda_2)$ and $I_3(\lambda_1', \lambda_2')$ are equivalent for all choices of parameters $p\in (0,\infty]$, $q\in (0,2]$, $\alpha\in \R$, $\beta \geq 0$ and $k\in \N_0$. This implies the equalities
$
\beta - (1+\lambda_1 + \lambda_2)\gamma  = \max\{\beta \lambda_1', (\beta-k)\lambda_1'\} - (1+\lambda_1' + \lambda_2')\gamma
$ and 
$
(\beta-k)\lambda_2 -  (1+\lambda_1 + \lambda_2)\gamma = (\beta-k)\lambda_2' -  (1+\lambda_1' + \lambda_2')\gamma.
$
for all $p\in (0,\infty]$, $q\in (0,2]$, $\alpha\in \R$, $\beta \geq 0$ and $k\in \N_0$ (the lower and upper bounds of the inequalities for $\alpha$ have to coincide). The second equality implies for $q=2$ ($\gamma = 0$) and $\beta-k=1 \neq 0$ the equality $\lambda_2 = \lambda_2' $. If we then cancel $(\beta-k)\lambda_2$ on both sides of the second equality and consider some $q \neq 2$ ($\gamma \neq 0$), we get $\lambda_1 = \lambda_1'$.

\item If $\lambda_1 \leq \lambda_2 \leq 1$ and $1\leq \lambda_1' \leq \lambda_2' $, we can argue as in case (2).

\item If $\lambda_1 \leq \lambda_2 \leq 1$ and $\lambda_1' \leq \lambda_2' \leq 1$, this entails that the inequalities $I_2(\lambda_1, \lambda_2)$ and $I_2(\lambda_1', \lambda_2')$ are equivalent for all choices of parameters $p\in (0,\infty]$, $q\in (0,2]$, $\alpha\in \R$, $\beta \geq 0$ and $k\in \N_0$. This implies the equalities
$
\max\{\beta \lambda_1, (\beta-k)\lambda_1\} - (1+\lambda_1 + \lambda_2)\gamma\\
= \max\{\beta \lambda_1', (\beta-k)\lambda_1'\} - (1+\lambda_1' + \lambda_2')\gamma
$and 
$
(\beta-k) -  (1+\lambda_1 + \lambda_2)\gamma = (\beta-k) -  (1+\lambda_1' + \lambda_2')\gamma.
$
for all $p\in (0,\infty]$, $q\in (0,2]$, $\alpha\in \R$, $\beta \geq 0$ and $k\in \N_0$ (the lower and upper bounds of the inequalities for $\alpha$ have to coincide). The first equality implies for $q=2$ ($\gamma = 0$) and $\beta\neq 0$ with $\beta-k \neq 0$ first that $\lambda_1$ and $\lambda_1'$ have the same sign and then that they are actually equal. The second equality then implies for $q\neq 2$ ($\gamma\neq 0$) the equality $\lambda_2 = \lambda_2'$. 

\item If $\lambda_1 \leq \lambda_2 \leq 1$ and $\lambda_1' \leq 1 \leq \lambda_2' $, this entails that the inequalities $I_2(\lambda_1, \lambda_2)$ and $I_3(\lambda_1', \lambda_2')$ are equivalent for all choices of parameters $p\in (0,\infty]$, $q\in (0,2]$, $\alpha\in \R$, $\beta \geq 0$ and $k\in \N_0$. This implies the equalities
$
\max\{\beta \lambda_1, (\beta-k)\lambda_1\} - (1+\lambda_1 + \lambda_2)\gamma\\  
= \max\{\beta \lambda_1', (\beta-k)\lambda_1'\} - (1+\lambda_1' + \lambda_2')\gamma
$and $
(\beta-k) -  (1+\lambda_1 + \lambda_2)\gamma = (\beta-k)\lambda_2' -  (1+\lambda_1' + \lambda_2')\gamma.
$
for all $p\in (0,\infty]$, $q\in (0,2]$, $\alpha\in \R$, $\beta \geq 0$ and $k\in \N_0$ (the lower and upper bounds of the inequalities for $\alpha$ have to coincide). The second equality implies for $q=2$ ($\gamma = 0$) and $\beta-k=1 \neq 0$ the equality $\lambda_2' = 1 $. The first equality implies for $q=2$ ($\gamma = 0$) and $\beta- k>0$ first that $\lambda_1'$ and $\lambda_1$ have the same sign and in a next step that they are equal. The second equality lastly implies for some $q\neq 2$ ($\gamma \neq 0$) that $\lambda_2 = \lambda_2'$.

\item If $\lambda_1 \leq 1\leq  \lambda_2 $ and $1\leq \lambda_1' \leq \lambda_2' $, we can argue as in case (3).

\item If $\lambda_1 \leq 1\leq  \lambda_2 $ and $\lambda_1' \leq \lambda_2' \leq 1 $, we can argue as in case (6).

\item If $\lambda_1 \leq 1 \leq \lambda_2$ and $\lambda_1' \leq 1 \leq \lambda_2'$, this entails that the inequalities $I_3(\lambda_1, \lambda_2)$ and $I_3(\lambda_1', \lambda_2')$ are equivalent for all choices of parameters $p\in (0,\infty]$, $q\in (0,2]$, $\alpha\in \R$, $\beta \geq 0$ and $k\in \N_0$. This implies the equalities
$
\max\{\beta \lambda_1, (\beta-k)\lambda_1\} - (1+\lambda_1 + \lambda_2)\gamma
= \max\{\beta \lambda_1', (\beta-k)\lambda_1'\} - (1+\lambda_1' + \lambda_2')\gamma
$ and
$(\beta-k)\lambda_2 -  (1+\lambda_1 + \lambda_2)\gamma = (\beta-k)\lambda_2' -  (1+\lambda_1' + \lambda_2')\gamma.$
for all $p\in (0,\infty]$, $q\in (0,2]$, $\alpha\in \R$, $\beta \geq 0$ and $k\in \N_0$ (the lower and upper bounds of the inequalities for $\alpha$ have to coincide). The second equality implies for $q=2$ ($\gamma = 0$) and $\beta-k=1 \neq 0$ the equality $\lambda_2' = \lambda_2 $. The first equality implies for $q=2$ ($\gamma = 0$) and $\beta- k>0$ first that $\lambda_1'$ and $\lambda_1$ have the same sign and in a next step that they are equal.
\end{enumerate}
\end{proof}

For Toeplitz shearlet groups, this implies the following.

\begin{corollary}
For $\delta, \delta' \in \R$, the groups $H^{\delta}$ and $H^{\delta'}$ have the same embedding behavior if and only if $\delta = \delta'$.
\end{corollary}

\begin{proof}
The group $H^{\delta}$ has the same embedding behavior as the group $H^{1-\delta, 1-2\delta}$ and $H^{\delta'}$ has the same embedding behavior as $H^{1-\delta', 1-2\delta'}$, according to Corollary \ref{Cor:EmbeddingBehavior1}. Hence, Corollary \ref{Cor:SameEmbeddingStandard} implies that $H^{\delta}$ and $H^{\delta'}$ have the same embedding behavior if and only if $\{1-\delta, 1-2\delta\} = \{1-\delta', 1-2\delta'\}$. If $1-\delta = 1-\delta'$, then $\delta= \delta'$ and if $1-\delta = 1-2\delta'$, then also $1-2\delta= 1-\delta'$, which implies $\delta=\delta'=0$. 
\end{proof}
\subsection{Influence of the group on the embedding behavior}
In the last section of this chapter, we apply the results in Section \ref{section:EmbeddingStandardShearlet} and Section \ref{section:EmbeddingToeplitzShearlet} to investigate which groups allow embeddings into Sobolev spaces of higher smoothness and what the influence of the different parameters on the existence of such an embedding is. Corollary \ref{Cor:EmbeddingBehavior1} allows us to focus on standard shearlet groups with $\lambda_1 \leq \lambda_2$.

The next result is an immediate consequence of Theorem \ref{Cor:EmbeddingStandardLong}.

\begin{corollary} \label{cor:embedding_2}
Let $p \in (0,\infty]$, $q \in (0,2]$, $k\in \N_0$, $\beta \geq 0$ with $p\leq q$ and $\lambda_1 \leq \lambda_2$.
Necessary and sufficient for the existence of $\alpha \in \R$ such that the embedding 
$\Co(\mathrm{L}^{p}_{v^{(\alpha, \beta)}}(\R^3\rtimes H^{\lambda_1, \lambda_2}))\hookrightarrow W^{k,q}(\R^3)$
holds is

\begin{enumerate}[i)]
	\item the inequality $\beta \leq (\beta-k)\lambda_2$ if $1\leq \lambda_1 \leq \lambda_2$,
    
	\item the inequality $ \max\{\beta \lambda_1, (\beta-k) \lambda_1\} \leq \beta-k$ if $\lambda_1 \leq \lambda_2 \leq 1$,
    
    \item   the inequality  $\max\{\beta \lambda_1, (\beta-k) \lambda_1\}\leq (\beta-k)\lambda_2$ if $ \lambda_1 \leq 1 \leq \lambda_2$.
	\end{enumerate}
\end{corollary}

\begin{proof}
If the embedding holds, then the respective inequality has to hold, according to Theorem \ref{Cor:EmbeddingStandardLong}. On the other hand, if the inequalities in $(1)$, $(2)$, or $(3)$ hold, then the choice 
$\alpha:= (\beta - k)\lambda_2 - (1+\lambda_1 + \lambda_2)(1/2-1/q)$
ensures the embedding in case $(1)$ and $(3)$ and the choice 
$\alpha:= \beta - k - (1+\lambda_1 + \lambda_2)(1/2-1/q)$
ensures the embedding in case $(2)$, according to Theorem \ref{Cor:EmbeddingStandardLong}.
\end{proof}

\begin{remark} \label{rem:embedding_2}
\begin{enumerate}[i)]
\item If $\beta= k = 0$, then all these inequalities are satisfied, which means that for arbitrary parameters with $p\leq q$, there exists always one $\alpha$ such that the embedding
$
\Co(\mathrm{L}^{p}_{v^{(\alpha, 0)}}(\R^3\rtimes H^{\lambda_1, \lambda_2}))\hookrightarrow L^q(\R^3)
$
holds.

\item Since none of these inequalities hold for $\lambda_1 = \lambda_2 = 1$ and $k>0$, we infer that coorbit spaces associated to shearlet groups with isotropic scaling matrices never allow an embedding into Sobolev spaces with nontrivial smoothness requirement ($k>0$).

\item If $\beta = k>0$, then no embedding exists for $\lambda_1>0$, but for every $ \lambda_1 \leq 0$.

\item If $\beta- k > 0$, then we can always choose in case $(1)$ a sufficiently large $\lambda_2$ such that the inequality is fulfilled, in case (2) a sufficiently small $\lambda_1>0$ or $\lambda_1 \leq 0$ such that the inequality is fulfilled and in case (3) a sufficiently large $\lambda_2$ or a sufficiently small $\lambda_1>0$ or $\lambda_1 \leq 0$ such that the inequality holds. To phrase this more informally, it is easier to find smoother functions in the space $\Co(\mathrm{L}^{p}_{v^{(\alpha, \beta)}}(\R^3\rtimes H^{\lambda_1, \lambda_2}))$ the more different the pair $(\lambda_1, \lambda_2)$ is from $(1,1)$. 
\end{enumerate}
\end{remark}

Next, we consider the case $p=q$. An application of Theorem \ref{Cor:EmbeddingStandardLong} for this choice of parameters leads to the next result.

\begin{corollary} \label{cor:embedding_3}
Let $p \in (0,2]$, $k\in \N_0$, $\alpha \in \R$, $\beta \geq 0$ and $\lambda_1 \leq \lambda_2$.
Necessary and sufficient for the embedding 
$\Co(\mathrm{L}^{p}_{v^{(\alpha, \beta)}}(\R^3\rtimes H^{\lambda_1, \lambda_2}))\hookrightarrow W^{k,p}(\R^3)$
are
\begin{enumerate}[i)]
	\item the inequalities
    $\beta \leq \alpha + (1+\lambda_1 + \lambda_2)\left(\frac{1}{2}-\frac{1}{p}\right) \leq (\beta-k)\lambda_2$
    if $1\leq \lambda_1 \leq \lambda_2$,
    
	\item the inequalities
    $\max\{\beta \lambda_1, (\beta-k) \lambda_1\} \leq \alpha + (1+\lambda_1 + \lambda_2)\left(\frac{1}{2}-\frac{1}{p}\right)\leq \beta-k$
    if $ \lambda_1 \leq \lambda_2 \leq 1$,
    
    \item   the inequalities
     $       \max\{\beta \lambda_1, (\beta-k) \lambda_1\}\leq \alpha + (1+\lambda_1 + \lambda_2)\left(\frac{1}{2}-\frac{1}{p}\right)\leq (\beta-k)\lambda_2$
    if $ \lambda_1 \leq 1 \leq \lambda_2$.
	\end{enumerate}
\end{corollary}

\begin{remark}\label{rem:SmallerSmoother}
\begin{enumerate}[i)]
\item In the first case, the right inequality leads to the condition
\begin{align*}
k \leq \beta - \left( \alpha + (1+\lambda_1 + \lambda_2)\left(\frac{1}{2} - \frac{1}{p}\right)\right)\lambda_2^\inv
\end{align*}
for $k$. If we fix all parameters except $p$ and $k$, this shows that for smaller $p$ this upper bound for $k$ increases, which means that the functions in $\Co(L^{p}_{v^{(\alpha, \beta)}}(\R^3\rtimes H^{\lambda_1, \lambda_2}))$ are smoother for smaller $p$.

\item In the second case with $\lambda_1>0$, the right inequality leads to the similar condition
\begin{align*}
k \leq \beta - \left( \alpha + (1+\lambda_1 + \lambda_2)\left(\frac{1}{2} - \frac{1}{p}\right)\right)
\end{align*}
for $k$. If we fix all parameters except $p$ and $k$, this shows that for smaller $p$ this upper bound for $k$ increases, which means that in this case the functions in $\Co(L^{p}_{v^{(\alpha, \beta)}}(\R^3\rtimes H^{\lambda_1, \lambda_2}))$ are smoother for smaller $p$ as well.

\item In the third case with $\lambda_1>0$, the right inequality leads to the same condition
\begin{align*}
k \leq \beta - \left( \alpha + (1+\lambda_1 + \lambda_2)\left(\frac{1}{2} - \frac{1}{p}\right)\right)\lambda_2^\inv
\end{align*}
for $k$. If we fix all parameters except $p$ and $k$, this shows that for smaller $p$ this upper bound for $k$ increases, which means that in this case the functions in $\Co(L^{p}_{v^{(\alpha, \beta)}}(\R^3\rtimes H^{\lambda_1, \lambda_2}))$ are smoother for smaller $p$ as well.
\end{enumerate}
\end{remark}

\section{Conclusion}
This paper provides a case study how the decomposition space approach can be systematically employed to study embedding properties of coorbit spaces. The complexity of the characterizations we obtained also highlight that even though the approach developed by Voigtlaender often reduces complex questions of containment between function spaces into a mere combinatorial problem, understanding the latter in more conceptual terms remains a challenge. 


\nocite{*}
\bibliographystyle{plain}
\bibliography{sample}
\end{document}